\documentclass[hidelinks,11 pt]{article} 
\usepackage[top=22mm, bottom=22mm, left=22mm, right=22mm]{geometry}
\usepackage{rotating}
\usepackage{amsmath,amsfonts,amssymb,amsthm,enumerate,hyperref,multicol, graphicx, array, caption, subcaption, color,makecell}
\hypersetup{colorlinks,citecolor=blue, linkcolor=blue}
\usepackage{adjustbox}
\usepackage{booktabs,siunitx}
\usepackage[dvipsnames]{xcolor}
\usepackage{multirow}
\usepackage{mathtools}
\usepackage{bigstrut}
\usepackage{float}
\usepackage{tabularx}
\usepackage{pgfplots}
\allowdisplaybreaks
\usepackage{titlesec}
\usepackage{enumitem}
\usepackage{authblk}

\usepackage{lipsum}
\usepackage{tikz}
\usetikzlibrary{shapes, arrows.meta, positioning}
\definecolor{orcidlogocol}{HTML}{A6CE39}
\newtheorem{theorem}{Theorem}[section]
\newtheorem{definition}{Definition}[section]

\newtheorem{lemma}{Lemma}[section]
\newtheorem{rmrk}{Remark}[section]

\numberwithin{equation}{section}
\usepackage{orcidlink}
\begin{document}
	{	{ \LARGE{\title{{\bf A modified double inertial subgradient extragradient algorithm for non-monotone variational inequality with applications}}}}
		\author[1]{Watanjeet Singh$^{\orcidlink{0000-0002-6456-357X}}$\thanks{Corresponding author: watanjeetsingh@gmail.com }}
		\author[1]{Sumit Chandok$^{\orcidlink{0000-0003-1928-2952}}$ }
		\affil[1]{\small{Department of Mathematics, Thapar Institute of Engineering and Technology, \protect\\ Patiala, 147001, Punjab, India\protect\\
				{watanjeetsingh@gmail.com}, {sumit.chandok@thapar.edu}}}
		\date{}
		\maketitle
		\begin{abstract}
			This paper presents a modified iterative approach to solve the variational inequality problem using the double inertial technique in the context of a real Hilbert space. Our iterative technique involves a projection onto a generalized half-space and a self-adaptive step-size rule which works without prior knowledge of the Lipschitz constant of the operator. We establish a weak convergence result for a variational inequality involving a non-monotone cost operator along with weak and strong convergence results for quasi-monotone and strongly pseudo-monotone operators, respectively. Under a simplified framework, linear convergence of the proposed method is also discussed. Additionally, we provide some numerical experiments to demonstrate the effectiveness of our iterative algorithm compared to previously established algorithms in solving real-world applications. Finally, we carry out a sensitivity analysis of our algorithm to demonstrate its effectiveness across various parameter settings.
		\end{abstract}
		\noindent
		{\bf Keywords:} Variational inequality, Non-monotone mapping, Weak and strong convergence.   \\
		\noindent
		{\bf 2020 Mathematics Subject Classification:} 47J20. 65K15. 47J25. 
		47H05. 
		\section{Introduction}
		Consider $C$ to be a nonempty, closed and convex subset of a real Hilbert space $H$. Then the classical variational inequality for an operator $\mathcal{F}: H \to H$ is of the form: find $x  \in C$ such that 
		\begin{align} \label{V1}
			\langle \mathcal{F}x,y-x\rangle \geq 0, ~\text{for all}~ y  \in C.
		\end{align}
		The solution set of \ref{V1} is denoted by $VI(C,\mathcal{F})$. This variational inequality problem was independently introduced by Fichera \cite{1} and Stampaccia \cite{2} to address the problems arising from mechanics. Let $S_{\mathcal{F}}$ be the solution set of the dual variational inequality problem, i.e. $S_{\mathcal{F}}=\{x \in C| \langle \mathcal{F}y,y-x\rangle \geq 0,~\text{for all}~y \in C\}$. This $S_{\mathcal{F}}$ is closed and convex subset of $C$, furthermore, since $C$ is convex and $\mathcal{F}$ is continuous, we have $S_{\mathcal{F}} \subset VI(C,\mathcal{F})$. Also, if $\mathcal{F}$ is pseudo-monotone and continuous, we have $S_{\mathcal{F}}=VI(C,\mathcal{F})$.

		The variational inequality problem, in short, VIP, has earned considerable attention in optimization and nonlinear analysis because the theory provides a unified and natural treatment for various mathematical problems. Some of these mathematical problems include network equilibrium problems, optimal control problems, oligopolistic market equilibrium problems, and image restoration problems.
		\par
		The earliest and simplest algorithm to tackle VIP is the projected gradient method. However, the implementation of this method is limited as it requires $\mathcal{F}$ to be strongly monotone or inverse strongly monotone. In 1976, Korpelevich \cite{3} introduced the extragradient method (EGM) to avoid the strict requirement of monotonicity in the setting of finite-dimensional Euclidean spaces, assuming the cost operator to be both monotone and $L$-Lipschitz continuous. Nonetheless, employing the EGM necessitates the calculation of two projections on the set $C$ in each iteration, which may pose challenges in computing projections during numerical experiments.
		\par
		To overcome this, He \cite{4} and Sun \cite{5} proposed the projection and contraction method (PCM). The PCM has received a great deal of interest because of its computational efficiency. Another valuable modification over EGM is the subgradient extragradient method (SEGM) proposed by Censor et al. \cite{6}. The SEGM operates by replacing the second projection onto set $C$ with a projection onto a particular constructible half-space. This adjustment enhances the method's efficiency as the projection onto half-space is explicitly defined. 
		\par
		In 2019, Dong et al. \cite{21} combined SEGM and PCM using two different step-sizes in each iteration. It has been demonstrated that the iterative scheme generated by merging these methods offers a great computational advantage. 
		Subsequently, numerous iterative algorithms have been developed by combining the advantages of EGM (or SEGM) and PCM (see \cite{24,23,22} and references cited therein).
		\par
		Polyak \cite{16} studied the convergence of an inertial extrapolation algorithm and demonstrated that it enhances the algorithm's convergence speed. Over the years, the inertial technique has attracted enormous interest among researchers as one of the important tools to accelerate the convergence of the algorithms. Numerical iterative methods incorporating multi-step inertial extrapolation can significantly accelerate convergence in solving optimization problems.  Recently, Yao et al. \cite{19} proposed an iterative scheme with double inertial steps to solve pseudo-monotone variational inequality, showing that it significantly improves the single inertial step, as one of the inertial parameters is allowed to be 1. Wang et al.  \cite{20} proposed a double inertial projection method for a quasi-monotone variational inequality.
        Recently, Li et al. \cite{li} proposed a double inertial subgradient extragradient method combined with projection-contraction method to solve quasi-monotone variational inequality. For more literature on the double inertial technique, we refer to \cite{Huang,wat,18,dv}, and the references therein.
        \par
      When the Lipschitz constant of the cost operator is unknown, classical methods may no longer be applicable. To address this issue, self-adaptive step-size rule and the Armijo line search rule have been adopted, allowing algorithms to operate without prior knowledge of the operator’s Lipschitz constant. Recently, Long et al.~\cite{long} proposed a single inertial projection algorithm incorporating a line search strategy, where uniform continuity was assumed instead of $L$-Lipschitz continuity. Tan et al.~\cite{tan} introduced a Mann-type subgradient extragradient method for pseudo-monotone variational inequalities, relying on a weaker form of Lipschitz continuity. However, Armijo-type line search schemes often require multiple projections onto the feasible set at each iteration, which increases the computational burden. To mitigate this drawback, Yang and Liu~\cite{yang} proposed a novel self adaptive step-size rule. Although their approach assumes $L$-Lipschitz continuity, it does not require explicit knowledge of the Lipschitz constant.
      \par
      In recent years, considerable attention has been given to variational inequality problems without monotonicity assumptions on the cost operator. In particular, Liu and Yang~\cite{liu} established a weak convergence result for non-monotone variational inequalities in 2019. Subsequently, Thong et al.~\cite{th11} proposed a relaxed two-step Tseng-type extragradient method for non-monotone variational inequality problems by incorporating an inertial technique. For further studies on non-monotone variational inequality problems, we refer the reader to~\cite{Newww, zzhou}.

		\section{Motivation}
		Variational inequalities in Hilbert spaces offer a powerful mathematical framework with broad applications to network flow, economics, engineering, and optimization. In economic contexts like the Nash-Cournot oligopoly model, we use variational inequalities for analyzing and optimizing competitive behaviors, enhancing the efficiency of oligopoly markets and ensuring better resource allocation and operational strategies. In the Nash-Cournot oligopolistic market equilibrium model, Murphy et al. \cite{13} examined how multiple firms can reach a stable state, or Nash equilibrium, in an oligopolistic market, where a small number of firms dominate the industry and each acts independently while supplying the same product. In this context, a Nash equilibrium means no firm can increase its profit by changing its production level while others keep theirs unchanged.  Murphy et al. \cite{13} provided a method to find this equilibrium by accounting for each firm’s costs and the overall market demand. The analysis focuses on determining how much each firm should produce, considering their production costs and the quantity they can sell, based on the total market supply. It was shown that under certain conditions, like predictable production costs and a clear relationship between price and demand, this equilibrium can be calculated. Harker \cite{14} framed a monotone variational inequality problem for this model. The detailed mathematical model of this problem is discussed in Section 5.
		
		\par
		In another scenario, suppose a traffic network comprising a finite set of nodes connected by directed edges. Consider
		$D$ is the set of all edges in the network and
		$W$ denotes the set of directed node pairs. Let $u=(i,j) \in W$, where $i$ and $j$ denote the origin and destination nodes, respectively. Then $A_{u}$ represents the set of all paths from node $i$ to node $j$. The set of all paths in the network is given by $T=\bigcup\limits_{u \in W} A_{u}$. For every path $t \in T$, let $x_{t}$ represents the number of vehicles on that path. Each $u \in W$ corresponds to a positive value $d_{u}$ indicating the flow demand from node $i$ to node $j$. The set of feasible flows, denoted by
		$C$ is defined as follows:
		\begin{align*}
			C=\{x: \sum \limits_{t \in A_{u}}x_{t}=d_{u},~\text{for all} ~u \in W; x_{t}\geq 0~ \text{for all}~ t \in T\}.
		\end{align*}
		Given the flow vector $x$, the traffic flow on each edge $d \in D$ is given by
		\begin{align*}
			h_{d}=\sum \limits_{u \in W}\sum \limits_{t \in A_{u}}
			\rho_{td}x_{t},
		\end{align*}
		where
		\begin{align*}
			\rho_{td}=\begin{array}{cc}
				\bigg{ \{ }& 
				\begin{array}{cc}
					1 & \text{if $d$ belongs to path $t$} \\
					0 & \text{otherwise}.
				\end{array}
			\end{array}
		\end{align*}
		The total cost for path $t$ is calculated as $\mathcal{F}_{t}(x)=\sum \limits_{d \in D}\rho_{td}k_{d}$, where $k_{d}$ is the expenses on $d \in D$. A feasible vector flow $x^{*} \in C$ with positive components is called the equilibrium if for all $u \in W$, we have
		\begin{align}\label{addin}
			\mathcal{F}_{q}(x^{*})=\min \limits_{t \in A_{u}}\mathcal{F}_{t}(x^{*}).
		\end{align}
		Solving \eqref{addin} is equivalent to solving the following classical variational inequality problem: Find $x^{*} \in C$ such that 
		\begin{align*}
			\langle \mathcal{F}x^{*},x-x^{*}\rangle \geq 0,
		\end{align*}
		for all $x \in C$. 
		\par
		
		Furthermore, in image restoration, variational inequalities provide robust techniques to recover original images from noisy versions, thus improving quality and accuracy. The idea is to convert the corresponding least square problem to a variational inequality problem. This framework leads to improved restoration outcomes and accurate representation of the original image.
		\par
		Motivated by the research going in this direction, we propose a new double inertial iterative algorithm by generalizing the SEGM combined with the PCM, along with a generalized non-monotonic step-size. The present work builds on and generalizes earlier studies in this area, resulting in a broader and more flexible framework for the treatment of non-monotone variational inequalities.
        Unlike other existing methods, our method involves a projection onto a generalized half-space, which extends many results in the literature. The key features of our paper are:
		\begin{enumerate}
			\item We examine two inertial extrapolation steps, where one inertia can be selected from $[0,1]$ and the other as close as possible to 1.
			\item A generalized non-monotonic step-size is considered, which works without prior knowledge of the Lipschitz constant.
			\item A projection onto a closed, convex set $C$ is calculated, followed by a projection onto the generalized half-space.
			\item We provide a weak convergence result for a non-monotone variational inequality problem.
            \item We also provide a comparison  of our double inertial iterative scheme with other well-known iterative schemes involving no inertial or single inertial parameter.
		\end{enumerate}
		Our paper is organized as follows. In Section 2, we recall some basic definitions and results that are required to understand the main results. In Section 3, under certain assumptions, we first establish a weak convergence result for a non-monotone variational inequality, and then present a weak convergence result for a variational inequality involving a quasi-monotone operator. Additionally, we give a strong convergence outcome for a variational inequality involving a strongly pseudo-monotone operator. We also discuss the linear convergence of the algorithm under a simplified framework. In Section 4, we demonstrate the efficiency of our algorithm in addressing real-life applications, including network equilibrium flow, market equilibrium models, and image restoration problems.

		\section{Preliminaries}
		In this section, we outline fundamental definitions and lemmas essential for understanding our main result. Throughout the paper, we denote $C$ as a nonempty, closed, and convex subset of the real Hilbert space $H$.  
		\begin{definition}
			An operator $\mathcal{F}: H \to H$ is said to be:
			\begin{enumerate}[label=(\roman*)]
				\item $L$-Lipschitz continuous if there exists a constant $L >0$ such that $\|\mathcal{F}x-\mathcal{F}y\| \leq L \|x-y\|$, for all $x,y \in H $.
				\item monotone if $\langle \mathcal{F}x-\mathcal{F}y,x-y\rangle \geq 0$, for all $x,y \in H $.
				\item pseudo-monotone if $\langle \mathcal{F}x,y-x \rangle \geq 0$ implies $\langle \mathcal{F}y,y-x\rangle \geq 0$, for all $x,y \in H $.
				\item quasi-monotone if $\langle \mathcal{F}x,y-x \rangle > 0$ implies $\langle \mathcal{F}y,y-x\rangle \geq 0$, for all $x,y \in H $.
				\item $k$-strongly pseudo-monotone if there exists $k >0$ such that $\langle \mathcal{F}x,y-x \rangle \geq 0$ implies $\langle \mathcal{F}y,y-x\rangle \geq k \|y-x\|^2$, for all $x,y \in H $.
			\end{enumerate}    
		\end{definition}

		For each point $x \in H $, there is a unique nearest point $P_{C}(x)$ in $C$, such that $$P_{C}(x)=\underset{y  \in C} {\arg \min} \{\|x - y\|\}.$$ 
		The mapping $P_{C}: H \to C$ is called a metric projection of $H$ onto $C$. It is known that the metric projection $P_{C}$ is nonexpansive. Some of the properties of metric projection are:
		\begin{align}
			\langle x-P_{C}(x),y-P_{C}(y)\rangle \leq 0; ~\text{for all}~x \in H , y \in C,
		\end{align}
		and 
		\begin{align}
			\|P_{C}(x)-P_{C}(y)\|^2 \leq \langle P_{C}(x)-P_{C}(y),x-y \rangle, ~\text{for all}~x,y \in H .
		\end{align}
		\begin{lemma} \label{lem2.1}
			The following results hold in $H$:
			\begin{enumerate} [label=(\roman*)]
				\item $\|x+y\|^2 =\|x\|^2+\|y\|^2+2 \langle x,y\rangle$, for all $x,y \in H$.
				\item $\|\alpha x +\beta y\|^2=\alpha (\alpha +\beta)\|x\|^2+\beta (\alpha +\beta)\|y\|^2-\alpha \beta \|x-y\|^2$, for all $x,y \in H $ and $\alpha,\beta \in \mathbb{R}$.
			\end{enumerate}
		\end{lemma} 
		\begin{lemma} \cite{Ye}
			If either
			\begin{enumerate} [label=(\roman*)]
				\item $\mathcal{F}$ is pseudo-monotone on $C$ and $VI(C,\mathcal{F})\neq \phi$,
				\item $\mathcal{F}$ is the gradient of $\mathcal{G}$, where $\mathcal{G}$ is a differential quasiconvex function on an open set $\mathfrak{K}$, $C \subset \mathfrak{K}$ and attains its global minimum on $C$,
				\item $\mathcal{F}$ is quasi-monotone on $C$, $\mathcal{F}\neq 0$ on $C$ and $C$ is bounded,
				\item $\mathcal{F}$ is quasi-monotone on $C$, $\mathcal{F} \neq 0$ on $C$ and there exists a positive number $t$, such that, for every $r \in C$ with $\|r\| \geq t$, there exists $w \in C$ such that $\|w\| \leq t$ and $\langle \mathcal{F}r,w-v \rangle \leq 0$,
				\item $\mathcal{F}$ is quasi-monotone on $C$, $intC$ is non-empty and there exists $v^{*} \in VI(C,\mathcal{F})$ such that $\mathcal{F}v^{*}\neq 0$.
			\end{enumerate}
			Then $S_{\mathcal{F}}$ is non-empty.
		\end{lemma}
		\begin{lemma} \cite{7} \label{lem2.2}
			Suppose $\eta_{n}$, $\nu_{n}$ and $\theta_{n}$ are sequences in $[0,+\infty)$ such that
			\begin{align*}
				\eta_{n+1}\leq \eta_{n}+\nu_{n}(\eta_{n}-\eta_{n-1})+\theta_{n}~\text{for all}~n \geq 1, \sum \limits_{n=1}^{\infty}\theta_{n}<+\infty,
			\end{align*}
			and there exists a real number $\nu$ with $0 \leq \nu_{n}\leq \nu <1$ for all $n \in \mathbb{N}$. Then the following hold:
			\begin{enumerate} [label=(\roman*)]
				\item $\sum \limits_{n=1}^{\infty}[\eta_{n}-\eta_{n-1}]_{+}<+\infty$, where $[u]_{+}=\max\{u,0\}$;
				\item there exists $\eta^{*} \in [0,+\infty)$ such that $\lim \limits_{n \to \infty}\eta_{n}=\eta^{*}$.
			\end{enumerate}
		\end{lemma}
		
		\begin{lemma}  \cite{8} \label{lem2.3}
			Consider $C$ is a nonempty subset of $H$ and a sequence $\{x_{n}\}$ in $H$ such that the following two conditions hold:
			\begin{enumerate} [label=(\roman*)]
				\item for each $u  \in C$, $\lim \limits_{n \to \infty}\|x_{n}-u\|$ exists;
				\item every sequential weak cluster point of $\{x_{n}\}$ is in $C$.
			\end{enumerate}
			Then $\{x_{n}\}$ converges weakly to a point in $C$.
		\end{lemma}

		\section{Main Results}
		In this section, we present our main algorithm and theorems. The following assumptions are required to prove our main results.
		\begin{enumerate}
			\item[(A1):] The mapping $\mathcal{F}$ is $L$-Lipschitz continuous on $H$ and is sequentially weakly continuous on $C$.
			\item [(A2):] $S_{\mathcal{F}}\neq \phi$.
			\item [(A3):] If $x_{n} \rightharpoonup v^{*}$ and $\limsup \limits_{n \to \infty}\langle \mathcal{F}x_{n},x_{n}\rangle \leq \langle \mathcal{F}v^{*},v^{*} \rangle$, then $\lim \limits_{n \to \infty}\langle \mathcal{F}x_{n},x_{n}\rangle=\langle \mathcal{F}v^{*},v^{*} \rangle$.
			\item [(A4):] $0 \leq \nu_{n} \leq \nu_{n+1}\leq 1$.
			\item [(A5):] $0 \leq \xi_{n} \leq \xi_{n+1}\leq \xi < \min\{\frac{\bar{\theta}-\sqrt{2\bar{\theta}}}{\bar{\theta}}, \nu_{1}\}$; $\bar{\theta} \in (2,+\infty)$.
			\item [(A6):] $0 < \alpha \leq \alpha_{n}\leq \alpha_{n+1}<\frac{1}{1+\bar{\theta}}$; $\bar{\theta} \in (2,+\infty)$.
			\item [(A7):] Let $\{\delta_{n}\} \in [1,+\infty)$ such that $\lim \limits_{n \to \infty}\delta_{n}=1$, $\{\chi_{n}\} \in [1,+\infty)$ such that $\sum \limits_{n=1}^{\infty}(\chi_{n}-1)<+\infty$ and $\{\zeta_{n}\} \in [0,+\infty)$ such that $\sum \limits_{n=0}^{\infty}\zeta_{n}<+\infty$.
		\end{enumerate}
		\begin{table}[ht]
			\centering
			\begin{tabular}{c}
				\hline
				\textbf{Algorithm 4.1 } Modified double inertial subgradient extragradient method (MDISEM) \\
				\hline
			\end{tabular}
		\end{table} 
		\textbf{Initialization:} Choose $\mu \in (0,1)$, $\lambda_{1}>0$, $\sigma \in (0,\frac{2}{\mu})$ and $\beta \in (\frac{\sigma}{2},\frac{1}{\mu})$. Choose $x_{0},x_{1} \in H $ and  calculate the iterate $x_{n+1}$ as:
		\\
		\textbf{Step 1.} Compute
		\begin{align}
			w_{n}&=x_{n}+\nu_{n}(x_{n}-x_{n-1}),\\
			y_{n}&=P_{C}(w_{n}-\beta \lambda_{n}\mathcal{F}w_{n})
		\end{align}

		and update $\lambda_{n+1}$, the step-size, as
		\begin{align} \label{stepsize}
			\lambda_{n+1}=\begin{array}{cc}
				\bigg{ \{ }& 
				\begin{array}{cc}
					\min\{ \frac{\mu \delta_{n} \|w_{n}-y_{n}\|}{\|\mathcal{F}w_{n}-\mathcal{F}y_{n}\|},\chi_{n}\lambda_{n}+\zeta_{n}
					\} & \text{if} ~\mathcal{F}w_{n}\neq \mathcal{F}y_{n}\\
					\chi_{n}\lambda_{n}+\zeta_{n} & \text{otherwise}.
				\end{array}
			\end{array}
		\end{align}
		We consider $y_{n}$ as a solution of (VIP) if $w_{n}=y_{n}$ (or $\mathcal{F}y_n=0$).  Otherwise \\
		\textbf{Step 2.} Compute 
		\begin{align*}
			u_{n}=P_{T_{n}}(w_{n}-\sigma \lambda_{n}d_{n}\mathcal{F}y_{n}),
		\end{align*}
		where
		\begin{align*}
			T_{n}=\{x \in H : \langle w_{n}-\beta \lambda_{n}\mathcal{F}w_{n}-y_{n},x-y_{n}\rangle \leq 0\},
		\end{align*}
		and
		\begin{align} \label{dn value}
			d_{n}=\frac{\langle w_{n}-y_{n}, \eta_{n}\rangle}{\|\eta_{n}\|^2},
		\end{align}
		\begin{align} \label{en value}
			\eta_{n}=w_{n}-y_{n}-\beta \lambda_{n}(\mathcal{F}w_{n}-\mathcal{F}y_{n}).
		\end{align}
		\textbf{ Step 3.} Compute
		\begin{align}
			v_{n}&=x_{n}+\xi_{n} (x_{n}-x_{n-1}),\\
			x_{n+1}&=(1-\alpha_{n})v_{n}+\alpha_{n}u_{n}.
		\end{align}

		Set $n \gets n+1$ and go to \textbf{Step 1.}
		\\
		We now present a flowchart for our Algorithm 4.1, illustrating its operational steps and process in \tablename~\ref{fig:algorithm_flowchart}.
		\begin{table}[ht!]
			\centering
			
			\begin{tikzpicture}[node distance=2cm, every node/.style={text centered}]
				
				\tikzstyle{startstop} = [rectangle, rounded corners, minimum width=3cm, minimum height=1cm, text centered, draw=black, fill=red!30]
				\tikzstyle{process} = [rectangle, minimum width=3cm, minimum height=1cm, text centered, draw=black, fill=orange!30]
				\tikzstyle{decision} = [diamond, minimum width=3cm, minimum height=1cm, text centered, draw=black, fill=green!30, aspect=2]
				\tikzstyle{arrow} = [thick,->,>=Stealth]
				
			
				\node (start) [startstop] {Start};
				\node (init) [process, below of=start] {Initialization: $\mu \in (0,1)$, $\lambda_{1}>0$, $\sigma \in (0,\frac{2}{\mu})$, $\beta \in (\frac{\sigma}{2},\frac{1}{\mu})$, $x_{0}$ and $x_{1}\in H$};
				\node (step1) [process, below of=init] {Step 1: Compute $w_{n}$ and  $y_{n}$};
				\node (update) [process, below of=step1] {Update $\lambda_{n+1}$ using \eqref{stepsize}};
				\node (decision1) [decision, below of=update, yshift=-1cm] {Is $w_{n}=y_{n}$ (or $\mathcal{F}y_n=0$)?};
				\node (stop) [startstop, below of=decision1, yshift=-1cm] {Stop};
				\node (step2) [process, right of=decision1, xshift=4cm, node distance=3.3cm] {Step 2: Compute $u_{n}$ using \eqref{dn value} and \eqref{en value}};
				\node (step3) [process, below of=step2, node distance=3cm] {Step 3: Compute $v_{n}$ and $x_{n+1}$};
				\node (next) [process, below of=step3, node distance=3cm] {Set $n \gets n+1$ and go to \textbf{Step 1} };
				
				\draw [arrow] (start) -> (init);
				\draw [arrow] (init) -> (step1);
				\draw [arrow] (step1) -> (update);
				\draw [arrow] (update) -> (decision1);
				\draw [arrow] (decision1) -- node[anchor=east] {Yes} (stop);
				\draw [arrow] (decision1.east) -- node[above] {No} (step2.west);
				\draw [arrow] (step2) -> (step3);
				\draw [arrow] (step3) -> (next);
				\draw [arrow] (next.east) -| ++(1,0) |- (step1.east);
			
			\end{tikzpicture}
			\caption{Flowchart of the Algorithm 4.1 with initialization and iterative steps.}
			\label{fig:algorithm_flowchart}
		\end{table}
		\begin{rmrk}
			\begin{enumerate} \label{remarks}
            \renewcommand{\labelenumi}{\roman{enumi}.}
            \item It is worth noting that Li et al. \cite{li} introduced a double inertial subgradient
extragradient method combined with the projection contraction method, incorporating a line
search rule, to address quasi-monotone variational inequalities. In contrast, we focus on the double inertial subgradient extragradient method combined with the projection contraction method to address non-monotone variational inequalities, without requiring any monotonicity assumption on the cost operator. Moreover, unlike the line search strategy adopted by Li et al. \cite{li}, our algorithm utilizes a self-adaptive step-size rule. Additionally, the second projection in our method is computed onto a generalized half-space.
				\item The step-size sequence is well defined under Assumption (A7), and the existence of $\lim \limits_{n \to \infty}\lambda_{n}$ is assured. Furthermore, it is straightforward to demonstrate that the sequence $\{\lambda_{n}\}$ possesses a lower bound, denoted as $\{\frac{\mu}{L}, \lambda_{1}\}$. For more details, we refer \cite{17}.
				\item We utilize the parameters $\beta$ and $\sigma$ to enhance the subgradient extragradient method and the projection and contraction method. Through numerical experiments, we aim to demonstrate that choosing appropriate values for $\beta$ and $\sigma$ can significantly improve both convergence speed and accuracy.
                \item 
Assumption $(A3)$ holds in several important cases (see \cite{liu}), that is,  
(a) when $x_n \rightharpoonup p^*$ and $\mathcal{F}$ is sequentially weakly continuous and monotone; 
and (b) when $x_n \rightharpoonup p^*$ and $\mathcal{F}$ is sequentially weakly--strongly continuous $(x_n \rightharpoonup p^* \implies \mathcal{F}x_n \to \mathcal{F}p^*)$.

			\end{enumerate}
		\end{rmrk}
        
		We now prove the weak convergence result using Algorithm 4.1 (MDISEM).
		\begin{lemma} \label{Lem4.1a}
			Assume that Assumptions (A1)-(A7) hold. Let $v^{*}$ be one of the weak cluster points of subsequence $\{w_{n_{j}}\}$ of $\{w_{n}\}$. If $\lim \limits_{j \to \infty}\|w_{n_{j}}-y_{n_{j}}\|=0$, then $v^{*} \in VI(C,\mathcal{F})$.
		\end{lemma}
		\begin{proof}
			We see that $w_{n_{j}} \rightharpoonup v^{*}$ and $\lim \limits_{j \to \infty}\|w_{n_{j}}-y_{n_{j}}\|=0$. It implies that $y_{n_{j}} \rightharpoonup v^{*}$ and since $y_{n} \in C$, we have $v^{*} \in C$. We now divide the proof into the following two cases.\par
			\textbf{Case 1:} If $\limsup \limits_{j \to \infty}\|\mathcal{F}y_{n_{j}}\|=0$, then we have $\lim \limits_{j \to \infty}\|\mathcal{F}y_{n_{j}}\|=\liminf \limits_{j \to \infty}\|\mathcal{F}y_{n_{j}}\|=0$. Then, from Assumption $(A1)$, we see that
			\begin{align*}
				0 < \|\mathcal{F}v^{*}\| \leq \liminf \limits_{j \to \infty}\|\mathcal{F}y_{n_{j}}\|=0.
			\end{align*}
			This means that $\mathcal{F}v^{*}=0$. \par
			\textbf{Case 2:} If $\limsup \limits_{j \to \infty}\|\mathcal{F}y_{n_{j}}\|>0$. Then, without loss of generality, we can assume that 
			$\limsup \limits_{j \to \infty}\|\mathcal{F}y_{n_{j}}\|=T_{1}>0$. Thus, we can find $j_{0}\geq 1$ such that $\|\mathcal{F}y_{n_{j}}\|>\frac{T_{1}}{2}$, for all $j \geq j_{0}$. We have
			\begin{align*}
				\langle w_{n_{j}}-\beta \lambda_{n_{j}}\mathcal{F}w_{n_{j}}-y_{n_{j}},x-y_{n_{j}}\rangle \leq 0~ \text{for all}~ x  \in C,
			\end{align*}
			which is further equivalent to
			\begin{align*}
				\frac{1}{\beta \lambda_{n_{j}}} \langle w_{n_{j}}-y_{n_{j}},x-y_{n_{j}} \rangle \leq \langle \mathcal{F}w_{n_{j}},x-y_{n_{j}} \rangle~ \text{for all}~ x  \in C.
			\end{align*}
			Thus, we have
			\begin{align} \label{newlem1a}
				\frac{1}{\beta \lambda_{n_{j}}} \langle w_{n_{j}}-y_{n_{j}},x-y_{n_{j}}\rangle+\langle \mathcal{F}w_{n_{j}},y_{n_{j}}-w_{n_{j}} \rangle \leq \langle \mathcal{F}w_{n_{j}},x-w_{n_{j}} \rangle~ \text{for all}~ x  \in C.
			\end{align}
			Since $\{w_{n_{j}}\}$ is weakly convergent, it is bounded. Therefore, from the Lipschitz continuity of $\mathcal{F}$, $\{\mathcal{F}w_{n_{j}}\}$ is also bounded. Also, as $\lim \limits_{k \to \infty}\|w_{n_{j}}-y_{n_{j}}\|=0$, we see that $\{y_{n_{j}}\}$ is also bounded. From the fact that $\{\lambda_{n}\}$ and $\beta$ are bounded, taking $j \to \infty$ in \eqref{newlem1a}, we have
			\begin{align} \label{newlem2a}
				\liminf \limits_{j \to \infty}\langle \mathcal{F}w_{n_{j}},x-w_{n_{j}} \rangle \geq 0,~\text{for all}~x  \in C. 
			\end{align}
			Also
			\begin{align} \label{newlem3a}
				\langle \mathcal{F}y_{n_{j}},x-y_{n_{j}} \rangle=\langle \mathcal{F}y_{n_{j}}-\mathcal{F}w_{n_{j}},x-w_{n_{j}} \rangle+\langle \mathcal{F}w_{n_{j}},x-w_{n_{j}} \rangle+\langle \mathcal{F}y_{n_{j}},w_{n_{j}}-y_{n_{j}}\rangle.
			\end{align}
			Since $\lim \limits_{j \to \infty}\|w_{n_{j}}-y_{n_{j}}\|=0$, and $\mathcal{F}$ is $L$-Lipschitz continuous on $H$, we get
			\begin{align*}
				\lim \limits_{n \to \infty}\|\mathcal{F}w_{n_{j}}-\mathcal{F}y_{n_{j}}\|=0,
			\end{align*}
			which together with \eqref{newlem2a} and \eqref{newlem3a} implies that
			\begin{align} \label{use1a}
				\liminf \limits_{j \to \infty}\langle \mathcal{F}y_{n_{j}},x-y_{n_{j}} \rangle \geq 0,~\text{for all}~x  \in C.
			\end{align}
			We choose a positive sequence $\{s_{n}\}$ such that $\lim \limits_{n \to \infty}s_{n}=0$, and 
			\begin{align*}
				\langle \mathcal{F}y_{n_{j}},x-y_{n_{j}}\rangle+s_{j}>0,~\text{for all}~j \geq 0.
			\end{align*}
			It implies that
			\begin{align} \label{star1}
				\langle \mathcal{F}y_{n_{j}},x \rangle+s_{j}>\langle \mathcal{F}y_{n_{j}},y_{n_{j}}\rangle,~\text{for all}~j\geq 0.
			\end{align}
			Setting $x\coloneqq v^{*}$ in \eqref{star1}, we get 
			\begin{align} \label{star2}
				\langle \mathcal{F}y_{n_{j}},v^{*} \rangle+s_{j}>\langle \mathcal{F}y_{n_{j}},y_{n_{j}}\rangle,~\text{for all}~j\geq 0.
			\end{align}
			Taking $j \to \infty$ in \eqref{star2} and from the fact that $y_{n_{j}} \rightharpoonup v^{*}$, we get
			\begin{align} \label{star3}
				\langle \mathcal{F}v^{*},v^{*}\rangle \geq \limsup \limits_{j \to \infty} \langle \mathcal{F}y_{n_{j}},y_{n_{j}}\rangle.
			\end{align}
			Using \eqref{star3} and Assumption (A3), we get
			\begin{align*}
				\lim \limits_{j \to \infty}\langle \mathcal{F}y_{n_{j}},y_{n_{j}}\rangle=\langle \mathcal{F}v^{*},v^{*}\rangle.
			\end{align*}
			Finally, from \eqref{star1}, we get
			\begin{align*}
				\langle \mathcal{F}v^{*},x\rangle&=\lim \limits_{j \to \infty}(\langle \mathcal{F}y_{n_{j}},x \rangle+s_{j}) \\
				&\geq \liminf \limits_{j \to \infty}\langle \mathcal{F}y_{n_{j}},y_{n_{j}} \rangle \\
				&=\lim \limits_{j \to \infty}\langle \mathcal{F}y_{n_{j}},y_{n_{j}}\rangle \\
				&=\langle \mathcal{F}v^{*},v^{*}\rangle.
			\end{align*}
			Thus, we get $\langle \mathcal{F}v^{*},x-v^{*} \rangle \geq 0$, ~\text{for all}~$x \in C$. Hence $v^{*} \in VI(C,\mathcal{F})$.
			
		\end{proof}
		\begin{lemma} \label{lem4.2a}
			Assume that the Assumptions (A1)-(A7) hold. Let $\{x_{n}\}$ be the sequence obtained by Algorithm 4.1. Then the following results hold:
			\begin{enumerate}
				\item[(i)] $\{x_{n}\}$ is bounded.
				\item[(ii)] $\lim \limits_{n \to \infty}\|x_{n+1}-x_{n}\|=0$.
				\item[(iii)] $\lim \limits_{n \to \infty}\|x_{n}-p^{*}\|$ exists for all $p^{*} \in S_{\mathcal{F}}$.
				\item [(iv)] $\lim \limits_{n \to \infty}\|x_{n}-w_{n}\|=0$.
			\end{enumerate}
		\end{lemma}
		\begin{proof} 
		\textbf{(i):} First we have to prove that $\{x_{n}\}$ is bounded. From the definition of $u_{n}$ and property of projection, we have for any $p^{*} \in S_{\mathcal{F}}$
			\begin{align} \label{T1}
				\|u_{n}-p^{*}\|^2&=\|P_{T_{n}}(w_{n}-\sigma \lambda_{n}d_{n}\mathcal{F}y_{n})-p^{*}\|^2\notag \\
				&\leq \|w_{n}-\sigma \lambda_{n} d_{n}\mathcal{F}y_{n}-p^{*}\|^2-\|w_{n}-\sigma \lambda_{n}d_{n}\mathcal{F}y_{n}-u_{n}\|^2\notag \\
				&=\|w_{n}-p^{*}\|^2+\|\sigma \lambda_{n}d_{n}\mathcal{F}y_{n}\|^2+2\langle w_{n}-p^{*},-\sigma \lambda_{n}d_{n}\mathcal{F}y_{n} \rangle-\|w_{n}-u_{n}\|^2\notag \\
				&-\|\sigma \lambda_{n} d_{n}\mathcal{F}y_{n} \|^2-2\langle w_{n}-u_{n},-\sigma \lambda_{n} d_{n} \mathcal{F}y_{n} \rangle \notag \\
				&=\|w_{n}-p^{*}\|^2-2\sigma \lambda_{n} d_{n} \langle w_{n}-p^{*},\mathcal{F}y_{n} \rangle-\|w_{n}-u_{n}\|^2+2\sigma \lambda_{n} d_{n} \langle w_{n}-u_{n},\mathcal{F}y_{n} \rangle \notag \\
				&=\|w_{n}-p^{*}\|^2-2\sigma \lambda_{n} d_{n} \langle \mathcal{F}y_{n},w_{n}-u_{n}+u_{n}-p^{*} \rangle-\|w_{n}-u_{n}\|^2+2\sigma \lambda_{n}d_{n} \langle w_{n}-u_{n},\mathcal{F}y_{n}\rangle \notag \\
				&=\|w_{n}-p^{*}\|^2-2\sigma \lambda_{n} d_{n} \langle \mathcal{F}y_{n},u_{n}-p^{*} \rangle-\|w_{n}-u_{n}\|^2.
			\end{align}
			Since $y_{n} \in C$ and $p^{*} \in S_{\mathcal{F}}$, we get
			\begin{align*}
				\langle \mathcal{F}y_{n},y_{n}-p^{*} \rangle \geq 0.
			\end{align*}
			This can be rewritten as
			\begin{align*}
				\langle \mathcal{F}y_{n},y_{n}-u_{n}+u_{n}-p^{*}\rangle \geq 0,
			\end{align*}
			which implies
			\begin{align} \label{eqm}
				\langle \mathcal{F}y_{n},p^{*}-u_{n}\rangle \leq \langle \mathcal{F}y_{n},y_{n}-u_{n} \rangle.
			\end{align}
			Using \eqref{dn value} and  \eqref{en value}, we have
			\begin{align}\label{T3}
				d_{n}=\frac{\langle w_{n}-y_{n}, \eta_{n} \rangle}{\|\eta_{n}\|^2}&=\frac{\langle w_{n}-y_{n}, w_{n}-y_{n}-\beta \lambda_{n}(\mathcal{F}w_{n}-\mathcal{F}y_{n}) \rangle}{\|\eta_{n}\|^2} \notag\\
				&=\frac{\|w_{n}-y_{n}\|^2-\langle w_{n}-y_{n},\beta \lambda_{n}(\mathcal{F}w_{n}-\mathcal{F}y_{n})\rangle}{\|\eta_{n}\|^2}\notag\\
				&\geq \frac{\|w_{n}-y_{n}\|^2- \beta \lambda_{n}\|w_{n}-y_{n}\|\|\mathcal{F}w_{n}-\mathcal{F}y_{n}\|}{\|\eta_{n}\|^2}.
			\end{align}
			Using step-size rule \eqref{stepsize}, we have
			\begin{align} \label{ss1}
				\|\mathcal{F}w_{n}-\mathcal{F}y_{n}\| \leq \frac{\mu \delta_{n}}{\lambda_{n+1}}\|w_{n}-y_{n}\|.
			\end{align}
			Using \eqref{ss1} in \eqref{T3}, we get
			\begin{align} \label{TTs1}
				d_{n}\geq \frac{(1-\beta \mu \delta_{n}\frac{\lambda_{n}}{\lambda_{n+1}})\|w_{n}-y_{n}\|^2}{\|\eta_{n}\|^2}
			\end{align}
			Now, again from \eqref{en value} and \eqref{ss1}, we have 
			\begin{align} \label{T5}
				\|\eta_{n}\|&=\|w_{n}-y_{n}-\beta \lambda_{n}(\mathcal{F}w_{n}-\mathcal{F}y_{n})\| \notag \\
				&\leq \|w_{n}-y_{n}\|+\beta \lambda_{n}\|\mathcal{F}w_{n}-\mathcal{F}y_{n}\| \notag \\
				& \leq \|w_{n}-y_{n}\|+\beta \mu \delta_{n}\frac{\lambda_{n}}{\lambda_{n+1}}\|w_{n}-y_{n}\| \notag \\
				& \leq (1+\beta \mu \delta_{n}\frac{\lambda_{n}}{\lambda_{n+1}})\|w_{n}-y_{n}\|.
			\end{align}
			From the Assumption (A7) and Remark 4.1(1), we have $\lim \limits_{n \to \infty}\frac{\delta_{n}\lambda_{n}}{\lambda_{n+1}}=1$. Thus, there exists a constant $n_{0} \in \mathbb{N}$ such that
			$(1-\beta \mu \delta_{n}\frac{\lambda_{n}}{\lambda_{n+1}})>0$ for all $n \geq n_{0}.$ 
			Combining \eqref{TTs1} and \eqref{T5}, we get
			\begin{align} \label{T6}
				d_{n}\geq \frac{(1-\beta \mu \delta_{n}\frac{\lambda_{n}}{\lambda_{n+1}})}{(1+\beta \mu \delta_{n}\frac{\lambda_{n}}{\lambda_{n+1}})^2}>0, ~\text{for all}~ n \geq n_{0}.
			\end{align}
			Using the definition of $T_{n}$, and from the fact that $u_{n} \in T_{n}$, we get
			\begin{align*}
				\langle w_{n}-\beta \lambda_{n}\mathcal{F}w_{n}-y_{n},u_{n}-y_{n}\rangle \leq 0.
			\end{align*}
			It implies
			\begin{align} \label{T8}
				\langle w_{n}-y_{n}-\beta \lambda_{n}(\mathcal{F}w_{n}-\mathcal{F}y_{n}),u_{n}-y_{n}\rangle \leq \beta \lambda_{n} \langle \mathcal{F}y_{n},u_{n}-y_{n}\rangle.
			\end{align}
			From \eqref{eqm} and \eqref{T6}, we have
			\begin{align}\label{T7}
				-2\sigma \lambda_{n} d_{n} \langle \mathcal{F}y_{n},u_{n}-p^{*}\rangle &\leq -2\sigma \lambda_{n}d_{n}\langle \mathcal{F}y_{n},u_{n}-y_{n} \rangle.
			\end{align}
			Now from \eqref{dn value}, \eqref{en value}, \eqref{T7} and \eqref{T8}, we have
			\begin{align} \label{T9}
				-2\sigma \lambda_{n}d_{n}\langle \mathcal{F}y_{n},u_{n}-p^{*} \rangle &\leq -2\frac{\sigma}{\beta}d_{n}\langle \eta_{n}, u_{n}-y_{n} \rangle \notag \\
				&=-2\frac{\sigma}{\beta}d_{n} \langle \eta_{n}, w_{n}-y_{n}-w_{n}+u_{n}\rangle \notag \\
				&=-2\frac{\sigma}{\beta} d_{n}\langle \eta_{n}, w_{n}-y_{n}\rangle-2\frac{\sigma}{\beta} d_{n}\langle \eta_{n},-w_{n}+u_{n}\rangle \notag \\
				&=-2\frac{\sigma}{\beta} d_{n}\langle \eta_{n}, w_{n}-y_{n} \rangle+2\frac{\sigma}{\beta} d_{n}\langle \eta_{n},w_{n}-u_{n} \rangle \notag \\
				&=-2\frac{\sigma}{\beta}d_{n}^2\|\eta_{n}\|^2+2\frac{\sigma}{\beta}d_{n}\langle \eta_{n},w_{n}-u_{n} \rangle.
			\end{align}
			Using Lemma \ref{lem2.1}, we have
			\begin{align}\label{T10}
				2\frac{\sigma}{\beta}d_{n}\langle \eta_{n}, w_{n}-u_{n}\rangle=\|w_{n}-u_{n}\|^2+\frac{\sigma^2}{\beta^2}d_{n}^2\|\eta_{n}\|^2-\|w_{n}-u_{n}-\frac{\sigma}{\beta}d_{n}\eta_{n}\|^2.
			\end{align}
			From \eqref{T3} and \eqref{T5}, we get
			\begin{align} \label{T11}
				d_{n}^2\|\eta_{n}\|^2 &\geq d_{n}(1-\beta \mu \delta_{n}\frac{\lambda_{n}}{\lambda_{n+1}})\|w_{n}-y_{n}\|^2 \notag \\
				&\geq (1-\beta \mu \delta_{n}\frac{\lambda_{n}}{\lambda_{n+1}})^2\frac{\|w_{n}-y_{n}\|^4}{\|\eta_{n}\|^2} \notag \\
				&\geq \frac{(1-\beta \mu \delta_{n}\frac{\lambda_{n}}{\lambda_{n+1}})^2}{(1+\beta \mu \delta_{n}\frac{\lambda_{n}}{\lambda_{n+1}})^2}\|w_{n}-y_{n}\|^2.
			\end{align}
			Using \eqref{T9}, \eqref{T10} and \eqref{T11} in \eqref{T1}, we get
			\begin{align} \label{T55}
				\|u_{n}-p^{*}\|^2 &\leq \|w_{n}-p^{*}\|^2-\|w_{n}-u_{n}-\frac{\sigma}{\beta}d_{n}\eta_{n}\|^2\notag \\
				&-\frac{\sigma}{\beta^2}(2\beta-\sigma)\frac{(1-\beta \mu \delta_{n}\frac{\lambda_{n}}{\lambda_{n+1}})^2}{(1+\beta \mu \delta_{n}\frac{\lambda_{n}}{\lambda_{n+1}})^2}\|w_{n}-y_{n}\|^2, ~\text{for all}~n \geq n_{0}.
			\end{align}
			As $2\beta-\sigma >0$, we get
			\begin{align} \label{imp1}
				\|u_{n}-p^{*}\| \leq \|w_{n}-p^{*}\|~\text{for all}~n \geq n_{0}.
			\end{align}
			Using \eqref{imp1}, we have
			\begin{align} \label{T31}
				\|x_{n+1}-p^{*}\|^2&=\|(1-\alpha_{n})v_{n}+\alpha_{n}u_{n}-p^{*}\|^2 \notag \\
				&=\|(1-\alpha_{n})(v_{n}-p^{*})+\alpha_{n}(u_{n}-p^{*})\|^2 \notag \\
				&=(1-\alpha_{n})\|v_{n}-p^{*}\|^2+\alpha_{n}\|u_{n}-p^{*}\|^2-\alpha_{n}(1-\alpha_{n})\|v_{n}-u_{n}\|^2 \notag \\
				&\leq (1-\alpha_{n})\|v_{n}-p^{*}\|^2+\alpha_{n}\|w_{n}-p^{*}\|^2-\alpha_{n}(1-\alpha_{n})\|v_{n}-u_{n}\|^2, ~\text{for all}~n \geq n_{0}.
			\end{align}
			As
			\begin{align*}
				x_{n+1}=(1-\alpha_{n})v_{n}+\alpha_{n}u_{n},
			\end{align*}
			we get
			\begin{align}\label{T32}
				\|v_{n}-u_{n}\|=\frac{1}{\alpha_{n}}\|x_{n+1}-v_{n}\|,~\text{for all}~ n \geq 1.
			\end{align}
			Combining \eqref{T31} and \eqref{T32}, we get
			\begin{align} \label{T33}
				\|x_{n+1}-p^{*}\|^2 \leq (1-\alpha_{n})\|v_{n}-p^{*}\|^2+\alpha_{n}\|w_{n}-p^{*}\|^2-\frac{(1-\alpha_{n})}{\alpha_{n}}\|x_{n+1}-v_{n}\|^2,~\text{for all}~ n \geq n_{0}.
			\end{align}
			Using Lemma \ref{lem2.1}, we have
			\begin{align} \label{T34}
				\|w_{n}-p^{*}\|^2&=\|x_{n}+\nu_{n}(x_{n}-x_{n-1})-p^{*}\|^2 \notag \\
				&=\|(1+\nu_{n})(x_{n}-p^{*})-\nu_{n}(x_{n-1}-p^{*})\|^2 \notag \\
				&=(1+\nu_{n})\|x_{n}-p^{*}\|^2-\nu_{n}\|x_{n-1}-p^{*}\|^2 +\nu_{n}(1+\nu_{n})\|x_{n}-x_{n-1}\|^2.
			\end{align}
			Similarly
			\begin{align} \label{T35}
				\|v_{n}-p^{*}\|^2=(1+\xi_{n})\|x_{n}-p^{*}\|^2-\xi_{n} \|x_{n-1}-p^{*}\|^2+\xi_{n} (1+\xi_{n})\|x_{n}-x_{n-1}\|^2.
			\end{align}
			Furthermore
			\begin{align} \label{T36}
				\|x_{n+1}-v_{n}\|^2&=\|x_{n+1}-x_{n}-\xi_{n}(x_{n}-x_{n-1})\|^2 \notag \\
				&=\|x_{n+1}-x_{n}\|^2+\xi_{n}^2\|x_{n}-x_{n-1}\|^2-2\xi_{n} \langle x_{n+1}-x_{n},x_{n}-x_{n-1}\rangle \notag \\
				&\geq \|x_{n+1}-x_{n}\|^2+\xi_{n}^2\|x_{n}-x_{n-1}\|^2-2\xi_{n} \|x_{n+1}-x_{n}\|\|x_{n}-x_{n-1}\|\notag \\
				& \geq (1-\xi_{n})\|x_{n+1}-x_{n}\|^2+(\xi_{n}^2-\xi_{n})\|x_{n}-x_{n-1}\|^2.
			\end{align}
			Substituting \eqref{T34}, \eqref{T35} and \eqref{T36} in \eqref{T33}, we get
			\begin{align} \label{T44}
				\|x_{n+1}-p^{*}\|^2 &\leq (1-\alpha_{n})\big[ (1+\xi_{n})\|x_{n}-p^{*}\|^2-\xi_{n}\|x_{n-1}-p^{*}\|^2+\xi_{n}(1+\xi_{n})\|x_{n}-x_{n-1}\|^2\big] \notag \\
				&+\alpha_{n}\big[(1+\nu_{n})\|x_{n}-p^{*}\|^2-\nu_{n}\|x_{n-1}-p^{*}\|^2+\nu_{n}(1+\nu_{n})\|x_{n}-x_{n-1}\|^2        \big] \notag \\
				&-\frac{(1-\alpha_{n})}{\alpha_{n}}\big[(1-\xi_{n})\|x_{n+1}-x_{n}\|^2+(\xi_{n}^2-\xi_{n})\|x_{n}-x_{n-1}\|^2 \big] \notag \\
				&=(1-\alpha_{n})(1+\xi_{n})\|x_{n}-p^{*}\|^2-\xi_{n} (1-\alpha_{n})\|x_{n-1}-p^{*}\|^2+\xi_{n} (1+\xi_{n})(1-\alpha_{n})\|x_{n}-x_{n-1}\|^2 \notag \\
				&+\alpha_{n}(1+\nu_{n})\|x_{n}-p^{*}\|^2-\alpha_{n}\nu_{n}\|x_{n-1}-p^{*}\|^2+\alpha_{n}\nu_{n}(1+\nu_{n})\|x_{n}-x_{n-1}\|^2 \notag \\
				&-\frac{(1-\alpha_{n})(1-\xi_{n})}{\alpha_{n}}\|x_{n+1}-x_{n}\|^2-\frac{(1-\alpha_{n})(\xi_{n}^2-\xi_{n})}{\alpha_{n}}\|x_{n}-x_{n-1}\|^2 \notag \\
				&=\big(1+\alpha_{n}\nu_{n}+\xi_{n} (1-\alpha_{n}) \big)\|x_{n}-p^{*}\|^2-(\alpha_{n}\nu_{n}+\xi_{n}(1-\alpha_{n}))\|x_{n-1}-p^{*}\|^2 \notag \\
				&-\kappa_{n}\|x_{n+1}-x_{n}\|^2+\pi_{n}\|x_{n}-x_{n-1}\|^2,
			\end{align}
			where 
			\begin{align*}
				\kappa_{n}=\frac{(1-\alpha_{n})(1-\xi_{n})}{\alpha_{n}}
			\end{align*}
			and
			\begin{align*}
				\pi_{n}=(1-\alpha_{n})\xi_{n} (1+\xi_{n})+\alpha_{n}\nu_{n}(1+\nu_{n})-\frac{(1-\alpha_{n})(\xi_{n}^2-\xi_{n})}{\alpha_{n}}.
			\end{align*}
			We define 
			\begin{align} \label{omeg1con}
				\Omega_{n}=\|x_{n}-p^{*}\|^2-(\alpha_{n}\nu_{n}+\xi_{n} (1-\alpha_{n}))\|x_{n-1}-p^{*}\|^2+\pi_{n}\|x_{n}-x_{n-1}\|^2,~n \geq 1.
			\end{align}
			Then
			\begin{align} \label{Teqn}
				\Omega_{n+1}-\Omega_{n}&=\|x_{n+1}-p^{*}\|^2-(\alpha_{n+1}\nu_{n+1}+\xi_{n+1} (1-\alpha_{n+1}))\|x_{n}-p^{*}\|^2+\pi_{n+1}\|x_{n+1}-x_{n}\|^2 \notag \\
				&-\|x_{n}-p^{*}\|^2+(\alpha_{n} \nu_{n}+\xi (1-\alpha_{n}))\|x_{n-1}-p^{*}\|^2-\pi_{n}\|x_{n}-x_{n-1}\|^2. \notag \\
				&\leq (\alpha_{n}\nu_{n}+\xi_{n}(1-\alpha_{n})-\alpha_{n+1}\nu_{n+1}-\xi_{n+1}(1-\alpha_{n+1}))\|x_{n}-p^{*}\|^2 \notag \\
				&-\kappa_{n}\|x_{n+1}-x_{n}\|^2+\pi_{n+1}\|x_{n+1}-x_{n}\|^2 \notag \\
				&=(\alpha_{n}(\nu_{n}-\xi_{n})-(\nu_{n+1}-\xi_{n+1})\alpha_{n+1}-(\xi_{n+1}-\xi_{n}))\|x_{n}-p^{*}\|^2\notag \\
                &-\kappa_{n}\|x_{n+1}-x_{n}\|^2+\pi_{n+1}\|x_{n+1}-x_{n}\|^2.
			\end{align}
            By $0 \leq \xi_{n} \leq \xi_{n+1}<\nu_{1}\leq \nu_{n}\leq \nu_{n+1}$ and $0 <\alpha_{n} \leq \alpha_{n+1}<1$, we have $-\alpha_{n+1}(\nu_{n+1}-\xi_{n+1})\leq -\alpha_{n}(\nu_{n+1}-\xi_{n+1})$ and $-(\xi_{n+1}-\xi_{n}) \leq -\alpha_{n}(\xi_{n+1}-\xi_{n})$. Therefore, we get
            \begin{align*}
                \alpha_{n}(\nu_{n}-\xi_{n})-\alpha_{n+1}(\nu_{n+1}-\xi_{n+1})-(\xi_{n+1}-\xi_{n}) &\leq\alpha_{n}(\nu_{n}-\xi_{n})-\alpha_{n}(\nu_{n+1}-\xi_{n+1})-\alpha_{n}(\xi_{n+1}-\xi_{n})\\
                &=\alpha_{n}(\nu_{n}-\nu_{n+1})\\
                &\leq 0.
            \end{align*}
 Thus, from \eqref{Teqn}, we have
			\begin{align} \label{Teqn1}
				\Omega_{n+1}-\Omega_{n} &\leq -\kappa_{n}\|x_{n+1}-x_{n}\|^2+\pi_{n+1}\|x_{n+1}-x_{n}\|^2 \notag \\
				&=-(\kappa_{n}-\pi_{n+1})\|x_{n+1}-x_{n}\|^2.
			\end{align}
			Consider
			\begin{align} \label{Teqn3}
				\kappa_{n}-\pi_{n+1}&=\frac{(1-\xi_{n})(1-\alpha_{n})}{\alpha_{n}}-(1-\alpha_{n+1})\xi_{n+1} (1+\xi_{n+1})-\alpha_{n+1}\nu_{n+1}(1+\nu_{n+1})+\frac{(1-\alpha_{n+1})(\xi_{n+1}^2-\xi_{n+1})}{\alpha_{n+1}} \notag \\
				&\geq \frac{(1-\xi_{n})(1-\alpha_{n+1})}{\alpha_{n+1}}+\frac{(1-\alpha_{n+1})(\xi_{n+1}^2-\xi_{n+1})}{\alpha_{n+1}}-(1-\alpha_{n+1})\xi_{n+1} (1+\xi_{n+1})-\alpha_{n+1}\nu_{n+1}(1+\nu_{n+1})\notag\\
                &\geq \frac{(1-\alpha_{n+1})(\xi_{n+1}^2-\xi_{n+1}-\xi_{n}+1)}{\alpha_{n+1}}-2(1-\alpha_{n+1})-2\alpha_{n+1}\notag \\
                &\geq \frac{(1-\alpha_{n+1})(\xi_{n+1}^2-2\xi_{n+1}-1)}{\alpha_{n+1}}-2\notag \\
                &\geq \bar{\theta}(\xi_{n+1}^2-2\xi_{n+1}+1)\notag \\
				&=\bar{\theta}\xi^2-2\bar{\theta}\xi +\bar{\theta}-2.
			\end{align}
			From Assumption (A5), we have $\xi_{n+1}\leq \xi < \frac{\bar{\theta}-\sqrt{2\bar{\theta}}}{\bar{\theta}}$, which implies
            \begin{align*}
                \bar{\theta}\xi_{n+1}^2-2\bar{\theta}\xi_{n+1} +\bar{\theta}-2\geq \bar{\theta}\xi^2-2\bar{\theta}\xi +\bar{\theta}-2>0.
            \end{align*} 
			Thus, from \eqref{Teqn1} and \eqref{Teqn3}, we have
			\begin{align} \label{Teqn4}
				\Omega_{n+1}-\Omega_{n} \leq -r\|x_{n+1}-x_{n}\|^2,
			\end{align}
			where $r=\bar{\theta} \xi^2-2\bar{\theta} \xi +\bar{\theta}-2$. Therefore, the sequence $\{\Omega_{n}\}$ is non-increasing. From \eqref{omeg1con}, we have
			\begin{align*}
				\Omega_{n}&=\|x_{n}-p^{*}\|^2-(\alpha_{n}\nu_{n}+\xi_{n} (1-\alpha_{n}))\|x_{n-1}-p^{*}\|^2+\pi_{n}\|x_{n}-x_{n-1}\|^2 \notag \\
				&\geq \|x_{n}-p^{*}\|^2-(\alpha_{n}\nu_{n}+\xi_{n}(1-\alpha_{n}))\|x_{n-1}-p^{*}\|^2.
			\end{align*}
			This implies
			\begin{align*}
				\|x_{n}-p^{*}\|^2 &\leq (\alpha_{n}\nu_{n}+\xi_{n} (1-\alpha_{n}))\|x_{n-1}-p^{*}\|^2+\Omega_{n} \notag \\
				&\leq (\frac{1}{1+\bar{\theta}}+\xi (1-\alpha))\|x_{n-1}-p^{*}\|^2+\Omega_{n} \notag \\
				&=\varepsilon \|x_{n-1}-p^{*}\|^2+\Omega_{n} \notag \\
				&\vdots \notag \\
				&\leq \varepsilon^{n}\|x_{0}-p^{*}\|^2+(1+\varepsilon+\varepsilon^2+...+\varepsilon^{n-1})\Omega_{1} \notag \\
				&\leq \varepsilon^{n}\|x_{0}-p^{*}\|^2+\frac{1}{1-\varepsilon}\Omega_{1},
			\end{align*}
			where $\varepsilon=\frac{1}{1+\bar{\theta}}+\xi (1-\alpha) <1$. So, the sequence $\{\|x_{n}-p^{*}\|\}$ is bounded and hence $\{x_{n}\}$ is bounded. \\
			\textbf{(ii):}	Now, we show that $\lim \limits_{n \to \infty}\|x_{n+1}-x_{n}\|=0$. Consider
			\begin{align*}
				\Omega_{n+1}&=\|x_{n+1}-p^{*}\|^2-(\alpha_{n+1}\nu_{n+1}+\xi_{n+1} (1-\alpha_{n+1}))\|x_{n}-p^{*}\|^2+\pi_{n+1}\|x_{n+1}-x_{n}\|^2 \notag \\
				&\geq -(\alpha_{n+1}\nu_{n+1}+\xi_{n+1} (1-\alpha_{n+1}))\|x_{n}-p^{*}\|^2.
			\end{align*}
			Thus, we have
			\begin{align*}
				-\Omega_{n+1}&\leq (\alpha_{n+1}\nu_{n+1}+\xi_{n+1} (1-\alpha_{n+1}))\|x_{n}-p^{*}\|^2 \notag \\
				& \leq \varepsilon \|x_{n}-p^{*}\|^2 \notag \\
				& \leq \varepsilon^{n+1}\|x_{0}-p^{*}\|^2 +\frac{\varepsilon}{1-\varepsilon}\Omega_{1}.
			\end{align*}
			Also, from \eqref{Teqn4}, we have
			\begin{align*}
				r\|x_{n+1}-x_{n}\|^2 &\leq \Omega_{n}-\Omega_{n+1}.
			\end{align*}
			It implies
			\begin{align*}
				r \sum \limits_{n=1}^{t}\|x_{n+1}-x_{n}\|^2 &\leq \sum \limits_{n=1}^{t}(\Omega_{n}-\Omega_{n+1}) \notag \\
				&\leq \Omega_{1}-\Omega_{t+1} \notag \\
				&\leq \Omega_{1}+\varepsilon^{t+1}\|x_{0}-p^{*}\|^2+\frac{\varepsilon \Omega_{1}}{1-\varepsilon} \notag \\
				&=\varepsilon^{t+1}\|x_{0}-p^{*}\|^2+\frac{\Omega_{1}}{1-\varepsilon}.
			\end{align*}
			Therefore $\sum \limits_{n=1}^{\infty}\|x_{n+1}-x_{n}\|^2 \leq \frac{\Omega_{1}}{r(1-\varepsilon)}<+\infty$
			and hence $\lim \limits_{n \to \infty}\|x_{n+1}-x_{n}\|=0$.\\
			\textbf{(iii):}	Next, we have to show that $\lim \limits_{n \to \infty}\|x_{n}-p^{*}\|$ exists. From \eqref{T44}, we have
			\begin{align*}
				\|x_{n+1}-p^{*}\|^2 &\leq (1+\alpha_{n}\nu_{n}+\xi_{n} (1-\alpha_{n}))\|x_{n}-p^{*}\|^2-(\alpha_{n}\nu_{n}+\xi_{n}(1-\alpha_{n}))\|x_{n-1}-p^{*}\|^2 \notag \\
				&-\kappa_{n}\|x_{n+1}-x_{n}\|^2+\pi_{n}\|x_{n}-x_{n-1}\|^2 \notag \\
				&\leq \|x_{n}-p^{*}\|^2+(\alpha_{n}\nu_{n}+\xi_{n}(1-\alpha_{n}))\|x_{n}-p^{*}\|^2-(\alpha_{n}\nu_{n}+\xi_{n} (1-\alpha_{n}))\|x_{n-1}-p^{*}\|^2 \\
                &+\pi_{n}\|x_{n}-x_{n-1}\|^2 \notag \\
				&=\|x_{n}-p^{*}\|^2+(\alpha_{n}\nu_{n}+\xi_{n} (1-\alpha_{n}))(\|x_{n}-p^{*}\|^2-\|x_{n-1}-p^{*}\|^2)+\pi_{n}\|x_{n}-x_{n-1}\|^2.
			\end{align*}
			As
			\begin{align*}
				\pi_{n}&=(1-\alpha_{n})\xi_{n} (1+\xi_{n})+\alpha_{n}\nu_{n}(1+\nu_{n})-\frac{(1-\alpha_{n})(\xi_{n}^2-\xi_{n})}{\alpha_{n}} \notag \\
				&\leq (1-\alpha)\xi (1+\xi)+\frac{2}{1+\bar{\theta}}+\frac{(1-\alpha)}{4\alpha}.
			\end{align*}
			Thus, applying Lemma \ref{lem2.2}, we have $ \lim \limits_{n \to \infty}\|x_{n}-p^{*}\|$ exists.\\
			\textbf{(iv):}	Finally, we show that $\lim \limits_{n \to \infty}\|x_{n}-w_{n}\|=0.$ Consider
			\begin{align*}
				\|x_{n+1}-v_{n}\| \leq \|x_{n+1}-x_{n}\|+\xi_{n} \|x_{n}-x_{n-1}\| \to 0~\text{as}~n \to \infty.
			\end{align*}
			Again on similar lines
			\begin{align*}
				\|x_{n+1}-w_{n}\|\leq \|x_{n+1}-x_{n}\|+\nu_{n}\|x_{n}-x_{n-1}\|\to 0~\text{as}~ n \to \infty.
			\end{align*}
			Also
			\begin{align*}
				\|x_{n+1}-v_{n}\|&=\|(1-\alpha_{n})v_{n}+\alpha_{n}u_{n}-v_{n}\| \notag \\
				&=\alpha_{n}\|v_{n}-u_{n}\|\geq \alpha \|v_{n}-u_{n}\|,
			\end{align*}
			implies 
			\begin{align*}
				\lim \limits_{n \to \infty}\|v_{n}-u_{n}\|=0.
			\end{align*}
			Furthermore
			\begin{align*}
				\|w_{n}-v_{n}\|&=\|x_{n}+\nu_{n}(x_{n}-x_{n-1})-x_{n}-\xi_{n}(x_{n}-x_{n-1})\| \notag \\
				&\leq \nu_{n}\|x_{n}-x_{n-1}\|+\xi_{n} \|x_{n}-x_{n-1}\| \notag \\
				& \leq \|x_{n}-x_{n-1}\|+\alpha \|x_{n}-x_{n-1}\| \to 0~\text{as}~n \to \infty.
			\end{align*}
			This implies $\lim \limits_{n \to \infty}\|w_{n}-v_{n}\|=0.$
			Also, it is easy to see that
			$\lim \limits_{n \to \infty}\|w_{n}-u_{n}\|=0.$
			Now, from \eqref{T55}, we have
			\begin{align*}
				\frac{\sigma}{\beta^2}(2\beta-\sigma)\frac{(1-\beta\mu \delta_{n}\frac{\lambda_{n}}{\lambda_{n+1}})^2}{1+\beta\mu \delta_{n}\frac{\lambda_{n}}{\lambda_{n+1}})^2}\|w_{n}-y_{n}\|^2
				&+\|w_{n}-u_{n}-\frac{\sigma}{\beta}d_{n}\eta_{n}\|^2 \notag \\
				&\leq \|w_{n}-p^{*}\|^2-\|u_{n}-p^{*}\|^2 \notag \\
				&\leq (\|w_{n}-p^{*}\|+\|u_{n}-p^{*}\|)\|w_{n}-u_{n}\|.
			\end{align*}
			Therefore, we have $\lim \limits_{n \to \infty}\|w_{n}-y_{n}\|=0$ and $\lim \limits_{n \to \infty}\|w_{n}-u_{n}-\frac{\sigma}{\beta}d_{n}\eta_{n}\|=0.$ Also, it can be easily shown that $\lim \limits_{n \to \infty}\|x_{n}-v_{n}\|=0$, $\lim \limits_{n \to \infty}\|x_{n}-w_{n}\|=0$ and $\lim \limits_{n \to \infty}\|x_{n}-y_{n}\|=0$.
		\end{proof}
		Now, we prove our main theorem.
		\begin{theorem} \label{theorem1a}
			Assume that the Assumptions (A1)-(A7) hold and $VI(C,\mathcal{F})$ is a finite set. Then the sequence $\{x_{n}\}$ generated by Algorithm 4.1 converges weakly to an element in $VI(C,\mathcal{F})$.
		\end{theorem}
		\begin{proof}
			From Lemma \ref{lem4.2a}, we see that the sequence $\{x_{n}\}$ is bounded. Thus, we can choose a subsequence $\{x_{n_{j}}\}$ of $\{x_{n}\}$ such that $x_{n_{j}} \rightharpoonup v^{*}$ as $j \to \infty$. Again from Lemma \ref{lem4.2a}, we see that $\lim \limits_{n \to \infty}\|x_{n}-w_{n}\|=0$ and $\lim \limits_{n \to \infty}\|w_{n}-y_{n}\|$. From Lemma \ref{Lem4.1a}, it implies that every weak cluster point of $\{x_{n}\}$ is in $VI(C,\mathcal{F})$. From the assumption that $VI(C,\mathcal{F})$ is finite, we obtain a sequence $\{x_{n}\}$ has finite cluster points in $VI(C,\mathcal{F})$. As $\lim \limits_{n \to \infty}\|x_{n+1}-x_{n}\|=0$. It implies that there exists $N \in \mathbb{N}$ such that
			\begin{align*}
				\|x_{n+1}-x_{n}\|<p,~\text{for all}~n \geq N.
			\end{align*}
			Assume that $\{x_{n}\}$ has more than one weak cluster points in $VI(C,\mathcal{F})$. Then from \cite[Lemma 3.5]{liu}, there exists $N_{1}>N$ such that $x_{N_{1}} \in B^{i}$ and $x_{N_{1}+1} \in B^{j}$, where $i \neq j$, $i,j \in \{1,2,...,m\}$, $ m \geq 2$ and
			\begin{align*}
				B^{i}=\bigcap\limits_{j=1,j\neq i}^{m}\{v:\langle v,\frac{v^{i}-v^{j}}{\|v^{i}-v^{j}\|} \rangle >\frac{\|v^{i}\|^2-\|v^{j}\|^2}{2\|v^{i}-v^{j}\|}+p\},
			\end{align*}
			where $p\coloneqq \min_{\substack{i,j\in \{1,2,...,m\}}_{i\neq j}}\{\frac{\|v^{i}-v^{j}\|}{4}\},$ and $v^{1},v^{2},...,v^{m}$ are the finite weak cluster points of $\{x_{n}\}$. In particular, we have
			\begin{align} \label{31}
				\|x_{N_{1}+1}-x_{N_{1}}\|<p.
			\end{align}
			Now 
			\begin{align*}
				x_{N_{1}} \in B^{i}=\bigcap\limits_{j=1,j\neq i}^{m}\{v: \langle v,\frac{v^{i}-v^{j}}{\|v^{i}-v^{j}\|}\rangle>p+\frac{\|v^{i}\|^2-\|v^{j}\|^2}{2\|v^{i}-v^{j}\|}\},  
			\end{align*}
			and 
			\begin{align*}
				x_{N_{1}+1} \in B^{j}&=\bigcap\limits_{i=1,i\neq j}^{m}\{v:\langle v,\frac{v^{j}-v^{i}}{\|v^{j}-v^{i}\|}\rangle>p+\frac{\|v^{j}\|^2-\|v^{i}\|^2}{2\|v^{j}-v^{i}\|}\}\\
				&=\bigcap\limits_{i=1,i\neq j}^{m}\{v:\langle -v,\frac{v^{j}-v^{i}}{\|v^{j}-v^{i}\|}\rangle<-p+\frac{\|v^{i}\|^2-\|v^{j}\|^2}{2\|v^{i}-v^{j}\|}\}\\
				&=\bigcap\limits_{i=1,i\neq j}^{m}\{v:\langle -v,\frac{v^{i}-v^{j}}{\|v^{j}-v^{i}\|}\rangle>p+\frac{\|v^{j}\|^2-\|v^{i}\|^2}{2\|v^{i}-v^{j}\|}\}.\\
			\end{align*}
			Thus, we get
			\begin{align} \label{psi1}
				\langle x_{N_{1}},\frac{v^{i}-v^{j}}{\|v^{i}-v^{j}\|}\rangle>p+\frac{\|v^{i}\|^2-\|v^{j}\|^2}{2\|v^{i}-v^{j}\|}
			\end{align}
			and
			\begin{align} \label{psi2}
				\langle -x_{N_{1}+1},\frac{v^{i}-v^{j}}{\|v^{j}-v^{i}\|}\rangle>p+\frac{\|v^{j}\|^2-\|v^{i}\|^2}{2\|v^{i}-v^{j}\|}.
			\end{align}
			Adding \eqref{psi1}, \eqref{psi2} and using \eqref{31}, we get
			\begin{align*}
				2p < \langle x_{N_{1}}-x_{N_{1}+1},\frac{v^{i}-v^{j}}{\|v^{i}-v^{j}\|} \rangle \leq \|x_{N_{1}}-x_{N_{1}+1}\|<p.
			\end{align*}
			This is a contradiction. Hence, the sequence $\{x_{n}\}$ has only one weak cluster point $v^{*} \in VI(C,\mathcal{F})$. Thus, we see that $\{x_{n}\}$ converges weakly to an element of   $VI(C,\mathcal{F})$.
		\end{proof}
		Using our Algorithm 4.1 (MDISEM), we now give a weak convergence result involving $\mathcal{F}$ as quasi-monotone operator. Instead of Lemma \ref{Lem4.1a}, we use the following lemma which is as follows:
		\begin{lemma} \label{use3ab}
			Assume that Assumptions (A1)-(A2),(A4)-(A7) hold and $\mathcal{F}$ is quasi-monotone on $H$. Let $v^{*}$ be one of the weak cluster points of the subsequence $\{w_{n_{j}}\}$ of $\{w_{n}\}$. If $\lim \limits_{j \to \infty}\|w_{n_{j}}-y_{n_{j}}\|=0$, then $v^{*} \in S_{\mathcal{F}}$ or $\mathcal{F}v^{*}=0$.
		\end{lemma}
		\begin{proof}
			We see that $w_{n_{j}} \rightharpoonup v^{*}$ and $\lim \limits_{j \to \infty}\|w_{n_{j}}-y_{n_{j}}\|=0$. It implies that $y_{n_{j}} \rightharpoonup v^{*}$ and since $y_{n} \in C$, we have $v^{*} \in C$. We now divide the proof into the following two cases.\par
			\textbf{Case 1:} If $\limsup \limits_{j \to \infty}\|\mathcal{F}y_{n_{j}}\|=0$, then we have $\lim \limits_{j \to \infty}\|\mathcal{F}y_{n_{j}}\|=\liminf \limits_{j \to \infty}\|\mathcal{F}y_{n_{j}}\|=0$. Then, from Assumption $(A1)$, we see that
			\begin{align*}
				0 < \|\mathcal{F}v^{*}\| \leq \liminf \limits_{j \to \infty}\|\mathcal{F}y_{n_{j}}\|=0.
			\end{align*}
			This means that $\mathcal{F}v^{*}=0$. \par
			\textbf{Case 2:} If $\limsup \limits_{j \to \infty}\|\mathcal{F}y_{n_{j}}\|>0$. Then, without loss of generality, we can assume that 
			$\limsup \limits_{j \to \infty}\|\mathcal{F}y_{n_{j}}\|=T_{1}>0$. Thus, we can find $j_{0}\geq 1$ such that $\|\mathcal{F}y_{n_{j}}\|>\frac{T_{1}}{2}$, for all $j \geq j_{0}$. We have
			\begin{align*}
				\langle w_{n_{j}}-\beta \lambda_{n_{j}}\mathcal{F}w_{n_{j}}-y_{n_{j}},x-y_{n_{j}}\rangle \leq 0~ \text{for all}~ x  \in C,
			\end{align*}
			which is further equivalent to
			\begin{align*}
				\frac{1}{\beta \lambda_{n_{j}}} \langle w_{n_{j}}-y_{n_{j}},x-y_{n_{j}} \rangle \leq \langle \mathcal{F}w_{n_{j}},x-y_{n_{j}} \rangle~ \text{for all}~ x  \in C.
			\end{align*}
			Thus, we have
			\begin{align} \label{newlem1}
				\frac{1}{\beta \lambda_{n_{j}}} \langle w_{n_{j}}-y_{n_{j}},x-y_{n_{j}}\rangle+\langle \mathcal{F}w_{n_{j}},y_{n_{j}}-w_{n_{j}} \rangle \leq \langle \mathcal{F}w_{n_{j}},x-w_{n_{j}} \rangle~ \text{for all}~ x  \in C.
			\end{align}
			Since $\{w_{n_{j}}\}$ is weakly convergent, it is bounded. Therefore, from the Lipschitz continuity of $\mathcal{F}$, $\{\mathcal{F}w_{n_{j}}\}$ is also bounded. Also, as $\lim \limits_{k \to \infty}\|w_{n_{j}}-y_{n_{j}}\|=0$, we see that $\{y_{n_{j}}\}$ is also bounded. From the fact that $\{\lambda_{n}\}$ and $\beta$ are bounded, taking $j \to \infty$ in \eqref{newlem1}, we have
			\begin{align} \label{newlem2}
				\liminf \limits_{j \to \infty}\langle \mathcal{F}w_{n_{j}},x-w_{n_{j}} \rangle \geq 0,~\text{for all}~x  \in C. 
			\end{align}
			Also
			\begin{align} \label{newlem3}
				\langle \mathcal{F}y_{n_{j}},x-y_{n_{j}} \rangle=\langle \mathcal{F}y_{n_{j}}-\mathcal{F}w_{n_{j}},x-w_{n_{j}} \rangle+\langle \mathcal{F}w_{n_{j}},x-w_{n_{j}} \rangle+\langle \mathcal{F}y_{n_{j}},w_{n_{j}}-y_{n_{j}}\rangle.
			\end{align}
			Since $\lim \limits_{j \to \infty}\|w_{n_{j}}-y_{n_{j}}\|=0$, and $\mathcal{F}$ is $L$-Lipschitz continuous on $H$, we get
			\begin{align*}
				\lim \limits_{n \to \infty}\|\mathcal{F}w_{n_{j}}-\mathcal{F}y_{n_{j}}\|=0,
			\end{align*}
			which together with \eqref{newlem2} and \eqref{newlem3} implies that
			\begin{align} \label{use1}
				\liminf \limits_{j \to \infty}\langle \mathcal{F}y_{n_{j}},x-y_{n_{j}} \rangle \geq 0,~\text{for all}~x  \in C.
			\end{align}
			We now consider the following two subcases. \par
			\textbf{Case 2a:} If $\limsup \limits_{j \to \infty} \langle \mathcal{F}y_{n_{j}},x-y_{n_{j}}\rangle >0$. Then, we can choose a subsequence $\{y_{n_{j_{k}}}\}$ such that $\lim \limits_{k \to \infty}\langle \mathcal{F}y_{n_{j_{k}}},x-y_{n_{j_{k}}}\rangle >0$. Thus, there exists $k_{0} \geq 1$ such that
			\begin{align*}
				\langle \mathcal{F}y_{n_{j_{k}}},x-y_{n_{j_{k}}}\rangle >0,
			\end{align*}
			for all $k \geq k_{0}$. By quasi-monotonicity of $\mathcal{F}$ on $C$, we get
			\begin{align*}
				\langle \mathcal{F}x,x-y_{n_{j_{k}}}\rangle \geq 0.
			\end{align*}
			Letting $k \to \infty$ in the above inequality, we have $\langle \mathcal{F}x,x-v^{*}\rangle \geq 0$. It implies that $v^{*} \in S_{\mathcal{F}}$. \par
			\textbf{Case 2b:} If $\limsup \limits_{j \to \infty}\langle \mathcal{F}y_{n_{j}},x-y_{n_{j}} \rangle=0$. Then, by \eqref{use1}, we get 
			\begin{align*}
				\lim \limits_{j \to \infty}\langle \mathcal{F}y_{n_{j}},x-y_{n_{j}}\rangle=0,~\text{for all}~x \in C.
			\end{align*}
			It further implies that
			\begin{align*}
				\langle \mathcal{F}y_{n_{j}},x-y_{n_{j}}\rangle+|\langle \mathcal{F}y_{n_{j}},x-y_{n_{j}}\rangle|+\frac{1}{j+1}>0, ~\text{for all}~ x \in C.
			\end{align*}
			Let 
			\begin{align*}
				\tau_{j}=|\langle \mathcal{F}y_{n_{j}},x-y_{n_{j}}\rangle|+\frac{1}{j+1}>0,
			\end{align*}
			then we have
			\begin{align} \label{use2}
				\langle \mathcal{F}y_{n_{j}},x-y_{n_{j}}\rangle+\tau_{j}>0, ~\text{for all}~j \geq j_{0}. 
			\end{align}
			Furthermore, as $y_{n_{j}} \in C$, we have $\mathcal{F}y_{n_{j}}\neq 0$, for all $j \geq 1$. Setting $v_{n_{j}}=\frac{\mathcal{F}y_{n_{j}}}{\|\mathcal{F}y_{n_{j}}\|^2}$, we get $\langle \mathcal{F}y_{n_{j}},v_{n_{j}}\rangle=1$, for each $j \geq j_{0}$. It follows from \eqref{use2} that for each $j \geq j_{0}$
			\begin{align*}
				\langle \mathcal{F}y_{n_{j}},x+\tau_{j}v_{n_{j}}-y_{n_{j}}\rangle >0.
			\end{align*}
			Since $\mathcal{F}$ is quasi-monotone, we have
			\begin{align*}
				\langle \mathcal{F}(x+\tau_{j}v_{n_{j}}),x+\tau_{j}v_{n_{j}}-y_{n_{j}} \rangle \geq 0.
			\end{align*}
			Now, for all $j \geq j_{0}$, we finally have
			\begin{align*}
				\langle \mathcal{F}x,x+\tau_{j}v_{n_{j}}-y_{n_{j}}\rangle &= \langle \mathcal{F}x-\mathcal{F}(x+\tau_{j}v_{n_{j}}),x+\tau_{j}v_{n_{j}}-y_{n_{j}} \rangle+ \langle \mathcal{F}(x+\tau_{j}v_{n_{j}}),x+\tau_{j}v_{n_{j}}-y_{n_{j}} \rangle \\
				& \geq \langle \mathcal{F}x-\mathcal{F}(x+\tau_{j}v_{n_{j}}),x+\tau_{j}v_{n_{j}}-y_{n_{j}} \rangle \\
				& \geq -\|\mathcal{F}x-\mathcal{F}(x+\tau_{j}v_{n_{j}})\|\|x+\tau_{j}v_{n_{j}}-y_{n_{j}}\| \\
				&\geq -\tau_{j}L\|v_{n_{j}}\|\|x+\tau_{j}v_{n_{j}}-y_{n_{j}}\| \\
				&=-\tau_{j}L\frac{1}{\|\mathcal{F}v_{n_{j}}\|}\|x+\tau_{j}v_{n_{j}}-y_{n_{j}}\| \\
				& \geq -\tau_{j}L\frac{2}{T_{1}}\|x+\tau_{j}v_{n_{j}}-y_{n_{j}}\|.
			\end{align*}
			Since $\|x+\tau_{j}v_{n_{j}}-y_{n_{j}}\|$ is bounded and $\tau_{j} \to 0$ as $j \to \infty$. Thus, taking $j \to \infty$ in the inequality above, we have
			\begin{align}
				\langle \mathcal{F}x,x-v^{*}\rangle \geq 0,~\text{for all}~x \in C.
			\end{align}
			Hence $v^{*} \in S_{\mathcal{F}}$.

		\end{proof}
		\begin{theorem} \label{the22}
			Suppose that Assumptions (A1)-(A2),(A4)-(A7) hold, $\mathcal{F}$ is quasi-monotone on $H$ and $\mathcal{F}x\neq 0$ for all $x \in C$. Then the sequence $\{x_{n}\}$ obtained by MDISEM converges weakly to a point in $S_{\mathcal{F}} \subset VI(C,\mathcal{F})$.
		\end{theorem}
		\begin{proof}
			Using similar arguments as in Lemma \ref{lem4.2a}, we obtain that $\lim \limits_{n \to \infty} \|x_{n+1}-x_{n}\|=0$, $\lim \limits_{n \to \infty}\|x_{n}-p^{*}\|$ exist and $\lim \limits_{n \to \infty}\|x_{n}-w_{n}\|=0$.
			Suppose $t_{\omega}(x_{n})$ is the set of weak cluster points of $\{x_{n}\}$. Moreover, we see that $t_{\omega}(x_{n})=t_{\omega}(y_{n})=t_{\omega}(w_{n})$.
			We show that 
			\begin{align*}
				t_{\omega}(x_{n}) \subset S_{\mathcal{F}}.   
			\end{align*}
			Take $v^{*} \in t_{\omega}(x_{n})$, then there exists a subsequence $\{x_{n_{k}}\}$ of $\{x_{n}\}$ such that $x_{n_{k}} \rightharpoonup v^{*}$, as $k \to \infty$. Since $C$ is weakly closed, we have $v^{*} \in C$. As $\mathcal{F}x\neq 0$, for all $x \in \mathcal{C}$, we get $\mathcal{F}v^{*}\neq 0$. Therefore, it follows from Lemma \ref{use3ab} that $v^{*} \in S_{\mathcal{F}}$. By Lemma \ref{lem2.3}, we have $\{x_{n}\}$ converges weakly to a point in the solution set $S_{\mathcal{F}}$.
		\end{proof}
        \begin{rmrk}
           It is important to note that, when dealing with a non-monotone variational inequality, we impose the assumption that $VI(C,\mathcal{F})$ is a finite set. This assumption ensures that the sequence $\{x_{n}\}$ has only finitely many weak cluster points in $VI(C,\mathcal{F})$. In fact, it is shown that $\{x_n\}$ admits a unique weak cluster point in $VI(C,\mathcal{F})$, which yields weak convergence. In contrast, when dealing with a quasi-monotone variational inequality, we impose the condition $\mathcal{F}x \neq 0$ for all $x \in C$. Dropping this condition may still allow the existence of solutions in $VI(C,\mathcal{F})$; however, it no longer guarantees that the weak limit belongs to the solution set $S_{\mathcal{F}}$, since $ S_{\mathcal{F}} \subset VI(C,\mathcal{F})$  in general.
 
        \end{rmrk}
    Next, we establish the strong convergence result for our Algorithm 4.1 under the assumption that $\mathcal{F}$ is strongly pseudo-monotone and $L$-Lipschitz continuous on $H$. It has been shown that if $\mathcal{F}$ is strongly pseudo-monotone, then
		$VI(C,\mathcal{F})$ has a unique solution. Furthermore, $S_{\mathcal{F}}=VI(C,\mathcal{F})$.
		\begin{theorem}
			Suppose that Assumptions (A4)-(A7) hold, and $\mathcal{F}$ is $k$-strongly pseudo-monotone and $L$-Lipschitz continuous on $H$. Then the sequence $\{x_{n}\}$ generated by MDISEM converges strongly to unique element in $VI(C,\mathcal{F})$.
		\end{theorem}
		\begin{proof}
			We consider the unique element to be $p^{*} \in VI(C,\mathcal{F})$. This implies that $\langle \mathcal{F}p^{*},y_{n}-p^{*} \rangle \geq 0$, therefore by strong pseudo-monotonicity of $\mathcal{F}$, we get 
			$\langle \mathcal{F}y_{n},y_{n}-p^{*}\rangle \geq k\|y_{n}-p^{*}\|^2$.
			So $ \langle \mathcal{F}y_{n},y_{n}-u_{n}+u_{n}-p^{*}\rangle \geq k\|y_{n}-p^{*}\|^2$. This further implies
			\begin{align} \label{TTTT1}
				\langle \mathcal{F}y_{n},u_{n}-p^{*} \rangle &\geq k\|y_{n}-p^{*}\|^2+\langle \mathcal{F}y_{n},u_{n}-y_{n} \rangle \notag \\
				-2\sigma \lambda_{n}d_{n}\langle \mathcal{F}y_{n},u_{n}-p^{*}\rangle &\leq -2\sigma \lambda_{n}d_{n}\langle \mathcal{F}y_{n},u_{n}-y_{n} \rangle-2\sigma d_{n}\lambda_{n}k\|y_{n}-p^{*}\|^2.
			\end{align}
			Using \eqref{TTTT1} in \eqref{T1}, we get
			\begin{align}
				\|u_{n}-p^{*}\|^2 &\leq \|w_{n}-\sigma \lambda_{n}d_{n}\mathcal{F}y_{n}-p^{*}\|^2-\|w_{n}-\sigma \lambda_{n}d_{n}\mathcal{F}y_{n}-u_{n}\|^2 \notag \\
				&=\|w_{n}-p^{*}\|^2-\|w_{n}-u_{n}\|^2-2\sigma \lambda_{n}d_{n}\langle \mathcal{F}y_{n},u_{n}-p^{*} \rangle \notag \\
				&\leq \|w_{n}-p^{*}\|^2-\|w_{n}-u_{n}\|^2-2\sigma \lambda_{n}d_{n}\langle \mathcal{F}y_{n},u_{n}-y_{n}\rangle-2\sigma d_{n}\lambda_{n}k\|y_{n}-p^{*}\|^2.
			\end{align}
			From \eqref{T8} and using \eqref{T9} and \eqref{T10}, we get
			\begin{align*}
				-2\sigma d_{n}\lambda_{n}\langle \mathcal{F}y_{n},u_{n}-y_{n}\rangle &\leq -2\frac{\sigma}{\beta}d_{n}\langle \eta_{n},u_{n}-y_{n} \rangle \notag \\
				&=-2\frac{\sigma}{\beta}d_{n}^2\|\eta_{n}\|^2+\|w_{n}-u_{n}\|^2+\frac{\sigma^2}{\beta^2}d_{n}^2\|\eta_{n}\|^2-\|w_{n}-u_{n}-\frac{\sigma}{\beta}d_{n}\eta_{n}\|^2.
			\end{align*}
			Therefore, we have
			\begin{align}\label{TT11}
				\|u_{n}-p^{*}\|^2 &\leq \|w_{n}-p^{*}\|^2-\|w_{n}-u_{n}\|^2-2\sigma d_{n}\lambda_{n}k\|y_{n}-p^{*}\|^2-2\frac{\sigma}{\beta}d_{n}^2\|\eta_{n}\|^2+\|w_{n}-u_{n}\|^2 \notag \\
				&+\frac{\sigma^2}{\beta^2}d_{n}^2\|\eta_{n}\|^2-\|w_{n}-u_{n}-\frac{\sigma}{\beta}d_{n}\eta_{n}\|^2 \notag \\
				&=\|w_{n}-p^{*}\|^2-2\sigma d_{n}\lambda_{n}k\|y_{n}-p^{*}\|^2-\frac{\sigma}{\beta^2}d_{n}^2\|\eta_{n}\|^2(2\beta-\sigma)-\|w_{n}-u_{n}-\frac{\sigma}{\beta}d_{n}\eta_{n}\|^2.
			\end{align}
			Using \eqref{T11} in \eqref{TT11}, we get
			\begin{align*}
				\|u_{n}-p^{*}\|^2 &\leq \|w_{n}-p^{*}\|^2-\|w_{n}-u_{n}-\frac{\sigma}{\beta}d_{n}\eta_{n}\|^2-\frac{\sigma}{\beta^2}(2\beta-\sigma)\frac{(1-\beta \mu \delta_{n}\frac{\lambda_{n}}{\lambda_{n+1}})^2}{(1+\beta \mu \delta_{n}\frac{\lambda_{n}}{\lambda_{n+1}})^2}\|w_{n}-y_{n}\|^2 \notag \\
				&-2\sigma d_{n}\lambda_{n}k \|y_{n}-p^{*}\|^2. 
			\end{align*}
			Setting $\rho \leq \frac{(1-\beta \mu)^2}{(1+\beta \mu)^2}$, we get
			\begin{align*}
				1> \lim \limits_{n \to \infty} \frac{(1-\beta \mu \delta_{n}\frac{\lambda_{n}}{\lambda_{n+1}})^2}{(1+\beta \mu \delta_{n}\frac{\lambda_{n}}{\lambda_{n+1}})^2}=\frac{(1-\beta \mu)^2}{(1+\beta \mu)^2} \geq \rho,
			\end{align*}
			where $\rho$ is a fixed number. Also
			\begin{align*}
				d_{n}\geq \frac{(1-\beta \mu \delta_{n}\frac{\lambda_{n}}{\lambda_{n+1}})}{(1+\beta \mu \delta_{n}\frac{\lambda_{n}}{\lambda_{n+1}})^2}.
			\end{align*}
			Thus, we get
			\begin{align*}
				\lim \limits_{n \to \infty}d_{n} \geq \frac{1-\beta \mu}{(1+\beta \mu)^2}=d ~\text{(say)}.
			\end{align*}
			Hence, there exists $n_1\geq n_{0}$ such that for all $n \geq n_{1}$, we have
			\begin{align*}
				\|u_{n}-p^{*}\|^2 \leq \|w_{n}-p^{*}\|^2-\|w_{n}-u_{n}-\frac{\sigma}{\beta}d_{n}\eta_{n}\|^2-\frac{\sigma}{\beta^2}(2\beta-\sigma)\rho \|w_{n}-y_{n}\|^2-2\sigma d_{n}\lambda_{n}k\|y_{n}-p^{*}\|^2.
			\end{align*}
			Given that $\{\lambda_{n}\}$ has a lower bound $\min\{\frac{\mu}{L}, \lambda_{1}\} = \lambda$ (say), we get
			\begin{align} \label{un1}
				\|u_{n}-p^{*}\|^2 &\leq \|w_{n}-p^{*}\|^2-\frac{\sigma}{\beta^2}(2\beta-\sigma)\rho \|w_{n}-y_{n}\|^2-2\sigma d \lambda k \|y_{n}-p^{*}\|^2 \notag \\
				&=\|x_{n}+\nu_{n}(x_{n}-x_{n-1})-p^{*}\|^2-\frac{\sigma}{\beta^2}(2\beta-\sigma)\rho \|w_{n}-y_{n}\|^2-2\sigma d \lambda k \|y_{n}-p^{*}\|^2 \notag \\
				&=(1+\nu_{n})\|x_{n}-p^{*}\|^2-\nu_{n}\|x_{n-1}-p^{*}\|^2+\nu_{n}(1+\nu_{n})\|x_{n}-x_{n-1}\|^2\notag \\
				&-\frac{\sigma}{\beta^2}(2\beta-\sigma)\rho \|w_{n}-y_{n}\|^2-2\sigma d \lambda k \|y_{n}-p^{*}\|^2.
			\end{align}
			Substituting \eqref{un1} in \eqref{T31}, we get
			\begin{align*}
				\|x_{n+1}-p^{*}\|^2 &\leq (1-\alpha_{n})\|v_{n}-p^{*}\|^2+\alpha_{n}(1+\nu_{n})\|x_{n}-p^{*}\|^2-\alpha_{n}\nu_{n}\|x_{n-1}-p^{*}\|^2\notag \\
				&+\alpha_{n}\nu_{n}(1+\nu_{n})\|x_{n}-x_{n-1}\|^2 -\frac{\sigma}{\beta^2}(2\beta-\sigma)\rho \alpha_{n} \|w_{n}-y_{n}\|^2-2 \alpha_{n}\sigma d \lambda k \|y_{n}-p^{*}\|^2 \notag \\
				&-\alpha_{n}(1-\alpha_{n})\|u_{n}-v_{n}\|^2 \notag \\
				&=(1-\alpha_{n})\|v_{n}-p^{*}\|^2+\alpha_{n}(1+\nu_{n})\|x_{n}-p^{*}\|^2-\alpha_{n}\nu_{n}\|x_{n-1}-p^{*}\|^2\notag \\
				&+\alpha_{n}\nu_{n}(1+\nu_{n})\|x_{n}-x_{n-1}\|^2 -\frac{\sigma}{\beta^2}(2\beta-\sigma)\rho \alpha_{n} \|w_{n}-y_{n}\|^2-2 \alpha_{n}\sigma d \lambda k \|y_{n}-p^{*}\|^2 \notag \\
				&-\frac{(1-\alpha_{n})}{\alpha_{n}}\|x_{n+1}-v_{n}\|^2 \notag \\
				&\leq (1-\alpha_{n})\big[(1+\xi_{n})\|x_{n}-p^{*}\|^2-\xi_{n} \|x_{n-1}-p^{*}\|^2+\xi_{n}(1+\xi_{n})\|x_{n}-x_{n-1}\|^2\big] \notag \\
				&+\alpha_{n}(1+\nu_{n})\|x_{n}-p^{*}\|^2-\alpha_{n}\nu_{n}\|x_{n-1}-p^{*}\|^2+\alpha_{n}\nu_{n}(1+\nu_{n})\|x_{n}-x_{n-1}\|^2 \notag \\
				& -\frac{\sigma}{\beta^2}(2\beta-\sigma)\rho \alpha_{n} \|w_{n}-y_{n}\|^2-2 \alpha_{n}\sigma d \lambda k \|y_{n}-p^{*}\|^2 \notag \\
				&-\frac{(1-\alpha_{n})}{\alpha_{n}}\big[(1-\xi_{n})\|x_{n+1}-x_{n}\|^2+(\xi_{n}^2-\xi_{n})\|x_{n}-x_{n-1}\|^2 \big] \notag \\
				&=\big[(1-\alpha_{n})(1+\xi_{n})+\alpha_{n}(1+\nu_{n}) \big]\|x_{n}-p^{*}\|^2-(\xi_{n}(1-\alpha_{n})+\alpha_{n}\nu_{n})\|x_{n-1}-p^{*}\|^2 \notag \\
				&+\big[(1-\alpha_{n})\xi_{n}(1+\xi_{n})+\alpha_{n}\nu_{n}(1+\nu_{n})-\frac{(1-\alpha_{n})(\xi_{n}^2-\xi_{n})}{\alpha_{n}} \big]\|x_{n}-x_{n-1}\|^2 \notag \\
				&-\frac{\sigma}{\beta^2}(2\beta-\sigma)\rho \alpha_{n}\|w_{n}-y_{n}\|^2-2 \alpha_{n}\sigma d \lambda k\|y_{n}-p^{*}\|^2-\frac{(1-\alpha_{n})(1-\xi_{n})}{\alpha_{n}}\|x_{n+1}-x_{n}\|^2.  
			\end{align*}
			By the assumptions on $\alpha_{n},\nu_{n}$, and $\xi_{n}$, we get
			\begin{align*}
				(1-\alpha_{n})\xi_{n} (1+\xi_{n})+\alpha_{n}\nu_{n}(1+\nu_{n})-\frac{(1-\alpha_{n})(\xi_{n}^2-\xi_{n})}{\alpha_{n}} &\leq (1-\alpha)\xi (1+\xi)+\frac{2}{1+\bar{\theta}}-\frac{(1-\alpha)}{4\alpha} \notag \\
				&=K^{*}~\text{(say)}.
			\end{align*}
			Therefore, we have
			\begin{align*}
				\|x_{n+1}-p^{*}\|^2 &\leq (1+\alpha_{n}\nu_{n}+\xi_{n}(1-\alpha_{n}))\|x_{n}-p^{*}\|^2-(\alpha_{n}\nu_{n}+\xi_{n}(1-\alpha_{n}))\|x_{n-1}-p^{*}\|^2+K^{*}\|x_{n}-x_{n-1}\|^2 \notag \\
				&-2\alpha \sigma d \lambda k \|y_{n}-p^{*}\|^2.
			\end{align*}
			It implies
			\begin{align*}
				2\alpha \sigma d \lambda k \|y_{n}-p^{*}\|^2 &\leq \|x_{n}-p^{*}\|^2-\|x_{n+1}-p^{*}\|^2+(\alpha_{n}\nu_{n}+\xi_{n}(1-\alpha_{n})(\|x_{n}-p^{*}\|^2-\|x_{n-1}-p^{*}\|^2) \notag \\
				&+K^{*}\|x_{n}-x_{n-1}\|^2.
			\end{align*}
			Taking the summation, we get
			\begin{align*}
				2\alpha \sigma d \lambda k \sum \limits_{i=N}^{n}\|y_{i}-p^{*}\|^2 &\leq \|x_{N}-p^{*}\|^2-\|x_{n+1}-p^{*}\|^2+(\alpha_{n}\nu_{n}+\xi_{n}(1-\alpha_{n})\|x_{n}-p^{*}\|^2\notag \\
				&-(\alpha_{N-1}\nu_{N-1}+\xi_{N-1}(1-\alpha_{N-1}))\|x_{n-1}-p^{*}\|^2+K^{*}\sum \limits_{i=N}^{n}\|x_{i}-x_{i-1}\|^2.
			\end{align*}
			As, we have already shown that the sequence $\{x_{n}\}$ is bounded and also $\sum \limits_{i=N}^{\infty}\|x_{i}-x_{i-1}\|^2 <+\infty$, thus we get $\sum \limits_{i=N}^{\infty}\|y_{i}-p^{*}\|^2<+\infty$. Therefore, $\lim \limits_{n \to \infty}\|y_{n}-p^{*}\|=0$.
			Consequently, we get
			\begin{align*}
				\|x_{n}-p^{*}\|\leq \|x_{n}-w_{n}\|+\|w_{n}-y_{n}\|+\|y_{n}-p^{*}\| \to 0~\text{as}~n \to \infty.
			\end{align*}
			Hence the result.
		\end{proof}
       Now, we provide another algorithm from Algorithm~4.1 with fewer parameters by choosing $\nu_n = 1$, $\xi_n = 0$, $\delta_n = 1$, $\chi_n = 1$, $\zeta_n = 0$, $\gamma = 1$, $\sigma = 1$, and $\beta = 1$, which is straightforward to implement in practical applications. 
        \vspace{0.2in}
\begin{table}[ht]
			\centering
			\begin{tabular}{c}
				\hline
				\textbf{Algorithm 4.1a}\\
				\hline
			\end{tabular}
		\end{table} 
        
		\textbf{Initialization:} Choose $\mu \in (0,1)$, $\lambda_{1}>0$, $x_{0} ~\text{and}~x_{1} \in H $. Calculate the next iterate $x_{n+1}$ as:
		\\
		\textbf{Step 1.} Compute
		\begin{align*}
			w_{n}&=2x_{n}-x_{n-1},\\
			y_{n}&=P_{C}(w_{n}- \lambda_{n}\mathcal{F}w_{n})
		\end{align*}

		and update $\lambda_{n+1}$, the step-size, as
		\begin{align*}
			\lambda_{n+1}=\begin{array}{cc}
				\bigg{ \{ }& 
				\begin{array}{cc}
					\min\{ \frac{\mu  \|w_{n}-y_{n}\|}{\|\mathcal{F}w_{n}-\mathcal{F}y_{n}\|},\lambda_{n}
					\} & \text{if} ~\mathcal{F}w_{n}\neq \mathcal{F}y_{n}\\
					\lambda_{n}& \text{otherwise}.
				\end{array}
			\end{array}
		\end{align*}
		We consider $y_{n}$ as a solution of (VIP) if $w_{n}=y_{n}$ (or $\mathcal{F}y_n=0$).  Otherwise \\
		\textbf{Step 2.} Compute 
		\begin{align*}
			u_{n}=P_{T_{n}}(w_{n}- \lambda_{n}d_{n}\mathcal{F}y_{n}),
		\end{align*}
		where
		\begin{align*}
			T_{n}=\{x \in H : \langle w_{n}-\lambda_{n}\mathcal{F}w_{n}-y_{n},x-y_{n}\rangle \leq 0\},
		\end{align*}
		and
		\begin{align*} 
			d_{n}=\frac{\langle w_{n}-y_{n}, \eta_{n}\rangle}{\|\eta_{n}\|^2},
		\end{align*}
		\begin{align*} 
			\eta_{n}=w_{n}-y_{n}- \lambda_{n}(\mathcal{F}w_{n}-\mathcal{F}y_{n}).
		\end{align*}
		\textbf{ Step 3.} Compute
		\begin{align*}
			x_{n+1}&=(1-\alpha_{n})x_{n}+\alpha_{n}u_{n}.
		\end{align*}

		Set $n \gets n+1$ and go to \textbf{Step 1.}
        
\begin{rmrk}
    For Algorithm~4.1(a), Lemmas \ref{Lem4.1a}, \ref{lem4.2a}, and~\ref{use3ab} as well as Theorems~\ref{theorem1a} and~\ref{the22} continue to hold.
\end{rmrk}

We now establish the linear convergence of Algorithm~4.1 (MDISEM) by considering a constant step size $\lambda_n=\lambda\in\bigl(0,\tfrac{1}{L}\bigr)$ and parameters satisfying $\xi_{n}=0,~ \beta=1, ~\sigma=1,~ \gamma=1$, $\nu_{n}=\nu$ and $\alpha_{n}=\alpha$.\\

  \begin{center}
      \begin{tabular}{c}
				\hline
				\textbf{Algorithm 4.1b} Linear inertial subgradient extragradient method\\
				\hline
			\end{tabular}
  \end{center}
    
        
		\textbf{Initialization:} For $x_{0} ~\text{and}~x_{1} \in H $, the next iterate $x_{n+1}$ is calculated as follows:
		\\
		\textbf{Step 1.} Compute
		\begin{align*}
			w_{n}&=x_{n}+\nu(x_{n}-x_{n-1}),\\
			y_{n}&=P_{C}(w_{n}- \lambda \mathcal{F}w_{n})
		\end{align*}
		\textbf{Step 2.} Compute 
		\begin{align*}
			u_{n}=P_{T_{n}}(w_{n}- \lambda d_{n}\mathcal{F}y_{n}),
		\end{align*}
		where
		\begin{align*}
			T_{n}=\{x \in H : \langle w_{n}-\lambda\mathcal{F}w_{n}-y_{n},x-y_{n}\rangle \leq 0\},
		\end{align*}
		and
		\begin{align*} 
			d_{n}=\frac{\langle w_{n}-y_{n}, \eta_{n}\rangle}{\|\eta_{n}\|^2},
		\end{align*}
		\begin{align*} 
			\eta_{n}=w_{n}-y_{n}- \lambda(\mathcal{F}w_{n}-\mathcal{F}y_{n}).
		\end{align*}
		\textbf{ Step 3.} Compute
		\begin{align*}
			x_{n+1}&=(1-\alpha)x_{n}+\alpha u_{n}.
		\end{align*}

		Set $n \gets n+1$ and go to \textbf{Step 1.}
        \begin{theorem}
The sequence $\{x_{n}\}$ generated by Algorithm 4.1b converges linearly to a unique element in $VI(C,\mathcal{F})$, if  the following assumptions are satisfied:
\begin{enumerate} \item[(B1):] $\mathcal{F}$ is $k$-strongly pseudo-monotone and $L$-Lipschitz continuous;
\item[(B2):] The step-size $\lambda \in (0,\frac{1}{L})$; \item[(B3):] $0\leq \nu <\frac{1}{t}-1$, where $t=1-\frac{1}{2}\min\bigg\{\frac{(1-\lambda L)^2}{(1+\lambda L)^2},2\lambda k\frac{(1-\lambda L)}{(1+\lambda L)^2}\bigg\} \in (0,1)$; \item[(B4):] $0 <\alpha <\frac{1}{3}$. \end{enumerate}
\end{theorem}

        \begin{proof}
           Let $p^{*} \in VI(C,\mathcal{F})$ be the unique element. Then, we have
$\langle \mathcal{F}p^{*}, x - p^{*} \rangle \geq 0$, for all $x \in C$.
Since $\mathcal{F}$ is $k$-strongly pseudo-monotone, we have
$\langle \mathcal{F}y_n, y_n - p^{*} \rangle \geq k\|y_n - p^{*}\|^{2}$. 
It implies
\begin{align*}
\langle \mathcal{F}y_n, y_n - u_n \rangle + \langle \mathcal{F}y_n, u_n - p^{*} \rangle
\geq k\|y_n - p^{*}\|^{2}.
\end{align*}
Therefore, we have
\begin{align}
\langle \mathcal{F}y_n, u_n - p^{*} \rangle
\geq k\|y_n - p^{*}\|^{2} + \langle \mathcal{F}y_n, u_n - y_n \rangle .
\end{align}
On similar lines of \eqref{TT11}, we get
\begin{align} \label{re0}
\|u_n - p^{*}\|^{2}
\leq \|w_n - p^{*}\|^{2} - d_n^{2}\|\eta_n\|^{2}
- 2d_n \lambda k\| y_n - p^{*}\|^2 .
\end{align}
From the definition of $d_n$, we have
\begin{align} \label{re2}
d_n
&= \frac{\langle w_n - y_n, \eta_n \rangle}{\|\eta_n\|^{2}} \nonumber\\
&= \frac{\langle w_n - y_n, w_n - y_n - \lambda (\mathcal{F}w_n - \mathcal{F}y_n) \rangle}{\|\eta_n\|^{2}} \nonumber\\
&\geq \frac{(1 - \lambda L)\|w_n - y_n\|^{2}}{\|\eta_n\|^{2}} .
\end{align}
Consider
\begin{align} \label{re1}
\|\eta_n\|
&= \|w_n - y_n - \lambda(\mathcal{F}w_n - \mathcal{F}y_n)\| \nonumber\\
&\leq \|w_n - y_n\| + \lambda\|\mathcal{F}w_n - \mathcal{F}y_n\| \nonumber\\
&= (1 + \lambda L)\|w_n - y_n\|.
\end{align}
Using \eqref{re1} in \eqref{re2}, we get
\begin{align} \label{re3}
d_n \geq \frac{1 - \lambda L}{(1 + \lambda L)^{2}} .
\end{align}
Therefore,
\begin{align} \label{re4}
d_n^{2}\|\eta_n\|^{2}
\geq \frac{(1 - \lambda L)^{2}}{(1 + \lambda L)^{2}}
\|w_n - y_n\|^{2}.
\end{align}
Using \eqref{re3} and \eqref{re4} in \eqref{re0}, we get
\begin{align} \label{re5}
\|u_n - p^{*}\|^{2}
\leq \|w_n - p^{*}\|^{2}
- \frac{(1 - \lambda L)^{2}}{(1 + \lambda L)^{2}}\|w_n - y_n\|^{2}
- 2\lambda k \frac{(1 - \lambda L)}{(1 + \lambda L)^{2}}\|y_n - p^{*}\|^{2}.
\end{align}
As
\begin{align*}
\frac{(1 - \lambda L)^{2}}{(1 + \lambda L)^{2}}\|w_n - y_n\|^{2}
&+ 2\lambda k \frac{(1 - \lambda L)}{(1 + \lambda L)^{2}}\|y_n - p^{*}\|^{2}\\
&\geq \min\!\left\{
\frac{(1 - \lambda L)^{2}}{(1 + \lambda L)^{2}},
2\lambda k \frac{(1 - \lambda L)}{(1 + \lambda L)^{2}}
\right\}
(\|w_n - y_n\|^{2} + \|y_n - p^{*}\|^{2}) \\
&\geq \frac{1}{2}
\min\!\left\{
\frac{(1 - \lambda L)^{2}}{(1 + \lambda L)^{2}},
2\lambda k \frac{(1 - \lambda L)}{(1 + \lambda L)^{2}}
\right\}
\|w_n - p^{*}\|^{2}.
\end{align*}
Therefore, \eqref{re5} becomes
\begin{align} \label{re6}
\|u_n - p^{*}\|^{2} \leq t\|w_n - p^{*}\|^{2}.
\end{align}
From the definition of $x_n$, we have
\begin{align} \label{re7}
\|x_n - u_n\|^{2} = \frac{1}{\alpha^{2}}\|x_{n+1} - x_n\|^{2}.
\end{align}
Using \eqref{re6} and \eqref{re7}, we obtain
\begin{align*}
\|x_{n+1} - p^{*}\|^{2}
&\leq (1 - \alpha)\|x_n - p^{*}\|^{2}
+ \alpha t\|w_n - p^{*}\|^{2}
- \frac{1 - \alpha}{\alpha}\|x_{n+1} - x_n\|^{2} \\
&\leq (1 - \alpha)\|x_n - p^{*}\|^{2}
+ \alpha t[(1 + \nu)\|x_n - p^{*}\|^{2}
- \nu\|x_{n-1} - p^{*}\|^{2}
+ \nu(1 + \nu)\|x_n - x_{n-1}\|^{2}] \\
&\quad - \frac{1 - \alpha}{\alpha}\|x_{n+1} - x_n\|^{2}.
\end{align*}
It implies that
\begin{align*}
\|x_{n+1} - p^{*}\|^{2}
+ \frac{1 - \alpha}{\alpha}\|x_{n+1} - x_n\|^{2}
&\leq [1 - \alpha(1 - t(1 + \nu))]\|x_n - p^{*}\|^{2}
- \alpha t\nu\|x_{n-1} - p^{*}\|^{2}\\
&+ \nu(1 + \nu)\alpha t\|x_n - x_{n-1}\|^{2}.
\end{align*}
As $0 < \alpha < \frac{1}{3}$, we have
\begin{align*}
\|x_{n+1} - p^{*}\|^{2} + \|x_{n+1} - x_n\|^{2}
\leq [1 - \alpha(1 - t(1 + \nu))]
\left[
\|x_n - p^{*}\|^{2}
+ \frac{\nu(1 + \nu)\alpha t}{1 - \alpha(1 - t(1 + \nu))}
\|x_n - x_{n-1}\|^{2}
\right].
\end{align*}
Since
$\frac{\nu(1 + \nu)\alpha t}{1 - \alpha(1 - t(1 + \nu))} < 1$,
we get
\begin{align*}
\|x_{n+1} - p^{*}\|^{2} + \|x_{n+1} - x_n\|^{2}
\leq [1 - \alpha(1 - t(1 + \nu))]
\bigg[\|x_n - p^{*}\|^{2} + \|x_n - x_{n-1}\|^{2}\bigg].
\end{align*}
We define $b_n = \|x_n - p^{*}\|^{2} + \|x_n - x_{n-1}\|^{2}$, for all $n \geq 1$.
Thus, we have
\[
b_{n+1} \leq [1 - \alpha(1 - t(1 + \nu))] b_n.
\]
By induction, we get
\[
b_{n+1} \leq [1 - \alpha(1 - t(1 + \nu))]^n b_1.
\]
Thus, we have
\[
\|x_{n+1} - p^{*}\|^{2} \leq [1 - \alpha(1 - t(1 + \nu))]^n b_1.
\]
This completes the proof.
\end{proof}
\section{Numerical Illustrations}
In this section, we provide some numerical experiments to validate our main result. We compare the performance of our algorithm against some well-established algorithms of Thong et al. \cite{9}, Shehu et al. \cite{10}, Liu and Yang \cite{liu} and Thong et al. \cite{Newww}. 

Algorithm 4.1 (MDISEM) incorporates two inertial parameters, and we compare its performance with well-known algorithms that use no inertial parameter and those with a single inertial parameter (see \cite{liu, 10,Newww,9}).
For our convenience, we denote Algorithm 1 of Thong et al. \cite{9}, Algorithm 1 of Shehu et al. \cite{10}, Algorithm 3.1 of Liu and Yang \cite{liu} and Algorithm 4.1 of Thong et al. \cite{Newww} as T Algorithm, S Algorithm, L Algorithm and T1 Algorithm, respectively. 

We define $E_{n}=\|w_n-y_n\|$ to measure the $n$-th iteration error and the convergence of $E_{n} \to 0$ implies that the sequence $\{x_{n}\}$ converges to $p^{*}$. We terminate the iterative process for network equilibrium flow and Nash-Cournot problem if the error term, $E_{n}$, falls below $10^{-6}$. 

For the image restoration problem, we calculate the relative error $R_{n}$, and terminate the process if $R_{n}$ falls below a desired threshold $\epsilon$.

All the projections onto the set $C$ are calculated using Matlab inbuilt function \textit{quadprog}, which is a part of the optimization toolbox. The projection onto the half-space is calculated using an explicit formula (see \cite{He} for details).
All computations are performed using MATLAB 2018a on an Intel(R) Core(TM) i3-10110U CPU @ 2.10GHz computer with 8.00 GB of RAM. 
\subsection{Network equilibrium flow} \label{new1}
In this example, we deal with one of the most important problems in traffic networks, namely, the network equilibrium flow. Mastroeni and Pappalardo \cite{11} formulated the variational inequality model of this problem. We provide only a short description, for more details, we refer to \cite{11} and \cite{12}. The following notations will help understand the model:
\begin{enumerate}
    \item $x_{i}$ is the flow on the arc $P_{i}=(r,s)$, and $x=(x_{1},x_{2},...,x_{n})^{T}$ is the vector of the flow across all arcs;
    \item an upper bound $d_{i}$ is associated with each arc $P_{i}$ on its capacity, $d=(d_{1},d_{2},...,d_{n})$;
    \item for each arc $P_{i}$, $\mathcal{F}_{i}(x)$ is the cost-variation as a function of the flows, and $\mathcal{F}x=(\mathcal{F}_{1}(x),\mathcal{F}_{2}(x),...,\mathcal{F}_{n}(x))^{T}$, with $\mathcal{F}x\geq 0$;
    \item $r_{j}$ is the balance at the node $j$, $j=1,2,...,p$ and $r=(r_{1},r_{2},...,r_{q})^{T}$;
    \item $T=(a_{ij}) \in \mathbb{R}^{q} \times \mathbb{R}^{n}$ is the node-arc incidence matrix whose elements are
    \begin{align*}
        a_{ij}=\begin{array}{cc}
 \bigg{ \{ }& 
    \begin{array}{cc}
     -1 & \text{if $i$ is the initial node for the arc $P_{j}$} \\
      +1 & \text{if $i$ is the final node of the arc $P_{j}$}\\
      0 & \text{otherwise.}
    \end{array}
\end{array}
    \end{align*}   
\end{enumerate}
A flow $x$ is said to be a variational equilibrium flow for the capacitated model if and only if it solves the following variational inequality:
\begin{align}
    \text{find}~p^{*} \in C_{x}~\text{such that} ~ \langle \mathcal{F}p^{*},x-p^{*} \rangle \geq 0~\text{for all}~x \in C_{x},
\end{align}
where
\begin{align*}
    C_{x}=\{x \in \mathbb{R}^{n}, Tx=r,0\leq x \leq d \}.
\end{align*}
In the numerical experiments, we take $r=(-2,0,0,0,0,2)^{T}$ and $d=(2,1,1,1,1,1,2,2)$ and
\begin{align*}
    T=\begin{bmatrix}
-1 & -1 & 0 & 0 & 0 & 0 & 0 & 0\\
1 & 0 & -1 & -1 & 0 & 0 & 0 & 0\\
0 & 1 & 0 & 0 & -1 & -1 & 0 & 0 \\
0 & 0 & 1 & 0 & 1 & 0 & -1 & 0 \\
0 & 0 & 0 & 1 & 0 & 1 & 0 & -1 \\
0 & 0 & 0 & 0 & 0 & 0 & 1 & 1
\end{bmatrix}.
\end{align*}
The cost function is defined as $\mathcal{F}x = \text{diag}(D)x$, where $D = (5.5, 1, 2, 3, 4, 50, 3.5, 1.5)$. Now, it is clear that
\begin{align*}
    \langle \mathcal{F}x-\mathcal{F}y,x-y\rangle =\sum \limits_{i=1}^{8}D_{i}\|x_{i}-y_{i}\|^2 \geq 0.
\end{align*}
Thus, $\mathcal{F}$ is monotone, and it is easy to see that $\mathcal{F}$ is $\|D\|$-Lipschitz continuous. Also, if $x_{n} \to p^{*}$, then $\mathcal{F}x_{n} \to \mathcal{F}p^{*}$. Thus, it is easy to see from Remark \ref{remarks}(iv) that the cost operator $\mathcal{F}$ satisfies the Assumption $(A3)$. The solution to this problem is given by
\begin{align}
    p^{*}=(1.000,1.000,0.1575,0.8425,0.885,0.115,1.0425,0.9575)^{T}.
\end{align}
We use the following parameters for the computation:
\begin{itemize}
    \item \textbf{MDISEM:} $\mu=0.6$, $\lambda_{1}=0.6$, $\beta=0.8$, $\sigma=1.5$, $\alpha_{n}=0.5$, $\delta_{n}=1+\frac{1}{n}$, $\chi_{n}=1+\frac{1}{(n+1)^{1.1}}$, $\zeta_{n}=\frac{1}{(n+1)^{1.1}}$, $\xi=0.4990$ and $\nu_{n}=1$.
    \item \textbf{L Algorithm:} $\mu=0.6$, $\lambda_{0}=0.6$ and $p_{n}=\frac{1}{(n+1)^{1.1}}$.
    \item \textbf{S Algorithm:} $\mu=0.6$, $\lambda_{1}=0.6$, $\alpha_{n}=0.2$, $\nu_{n}=1$ and $\gamma=1.9$.
    \item \textbf{T Algorithm:} $\mu=0.6$, $\nu_{1}=0.6$, $\theta=0.2$, and $\alpha_{n}=\frac{1}{n^{\frac{3}{2}}}$.
    
\end{itemize}
Based on Figure \ref{trafgr} and Table \ref{traftab}, it is evident that the sequence produced by our algorithm converges to the solution much faster compared to previously established algorithms. Moreover, Figure \ref{permemb} and Table \ref{sen1} provide a detailed sensitivity analysis for varying $\mu,\beta$ and $\sigma$. \\
The performance of our algorithm in the context of network equilibrium flow demonstrates reliable efficiency and consistency across varying values of $\mu,\beta$ and $\sigma$. Through extensive testing on different parameters (see Table \ref{sen1}), the algorithm consistently exhibits fast convergence, ensuring computational efficiency. Moreover, the fluctuation in solution convergence remains minimal, further validating the robustness and stability of the proposed iterative scheme. It is also validated graphically in the Figure \ref{permemb} for $\mu=0.6$.
These results highlight the algorithm’s effectiveness in solving network equilibrium flow problems under diverse conditions.
\begin{figure}
    \centering
     \includegraphics[scale=0.5,trim=28mm 40mm 30mm 30mm]{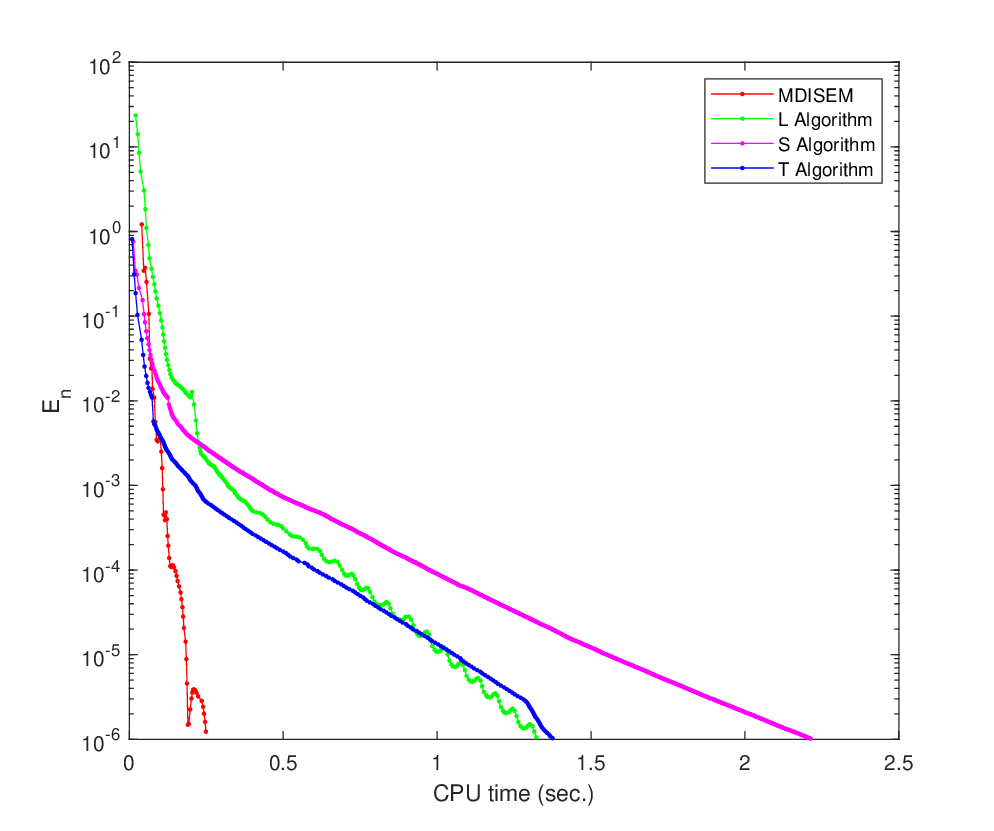}  
     \vspace{0.6in}
    \caption{Comparison of algorithms for network equilibrium flow.}
    \label{trafgr}
\end{figure}
\begin{table}[h!]
\centering
\begin{tabular}{ p{4cm} p{4cm} c }
\hline
Algorithm & CPU time (sec.)  & Iterations \\
\hline
MDISEM & 0.25 & 58\\
L Algorithm &1.32&213\\
S Algorithm   &2.22 & 605 \\
T Algorithm &1.38 & 205\\

\hline
\end{tabular}
\caption{Numerical results for network equilibrium flow.}
\label{traftab}
\end{table}

 \begin{sidewaystable}[]
 	
 \footnotesize
 \begin{tabular*}{\textwidth}{ccccccccccccc}
 \toprule
  & \multicolumn{4}{c}{$\sigma=1.8$} & \multicolumn{4}{c}{$\sigma=4.9$} & \multicolumn{4}{c}{$\sigma=5.6$} \\
 \cmidrule(lr){2-5}\cmidrule(lr){6-9}\cmidrule(lr){10-13}
 $\mu=0.2323$ & $\beta=1.4$ & $\beta=2.6$ & $\beta=3.1$ & $\beta=4.6$ & $\beta=2.5$ & $\beta=3.1$ & $\beta=3.9$ & $\beta=4.1$ & $\beta=2.9$ & $\beta=3.3$ & $\beta=3.7$ & $\beta=4.01$ \\
 \midrule
 Iterations & 56 & 88 & 126 & 199 & 74 & 65 & 87 & 189 & 49 & 59 & 64 & 81 \\
 \bottomrule
 &&&&&&&&&&&&\\
 \toprule
  & \multicolumn{4}{c}{$\sigma=0.49$} & \multicolumn{4}{c}{$\sigma=1.21$} & \multicolumn{4}{c}{$\sigma=2.44$} \\
 \cmidrule(lr){2-5}\cmidrule(lr){6-9}\cmidrule(lr){10-13}
 $\mu=0.3332$ & $\beta=0.30$ & $\beta=1.1$ & $\beta=2.6$ & $\beta=2.8$ & $\beta=0.8$ & $\beta=1.2$ & $\beta=2.2$ & $\beta=2.7$ & $\beta=1.23$ & $\beta=1.4$ & $\beta=2.6$ & $\beta=3$ \\
 \midrule
 Iterations & 90 & 160 & 571 & 729 & 56 & 70 & 141 & 232 & 60 & 44 & 76 & 187 \\
 \bottomrule
 &&&&&&&&&&&&\\
\toprule
  & \multicolumn{4}{c}{$\sigma=0.5$} & \multicolumn{4}{c}{$\sigma=1.8$} & \multicolumn{4}{c}{$\sigma=2.9$} \\
 \cmidrule(lr){2-5}\cmidrule(lr){6-9}\cmidrule(lr){10-13}
 $\mu=0.464$ & $\beta=0.3$ & $\beta=1.4$ & $\beta=1.9$ & $\beta=2.1$ & $\beta=1$ & $\beta=1.23$ & $\beta=1.96$ & $\beta=2.04$ & $\beta=1.56$ & $\beta=1.72$ & $\beta=1.89$ & $\beta=2.06$ \\
 \midrule
 Iterations & 76 & 217 & 413 & 624 & 59 & 47 & 82 & 119 & 50 & 55 & 47 & 71 \\
 \bottomrule
 
 \end{tabular*}
 \label{table:sigma_beta1}

 \caption{Sensitivity analysis of Algorithm 4.1 (MDISEM) in network equilibrium flow for different values of $\sigma$ and $\beta$.}
 \label{sen1}
 \end{sidewaystable}

\begin{figure}[] 
	\begin{subfigure}{0.49\textwidth}
		\includegraphics[scale=0.5,trim=28mm 60mm 20mm 80mm]{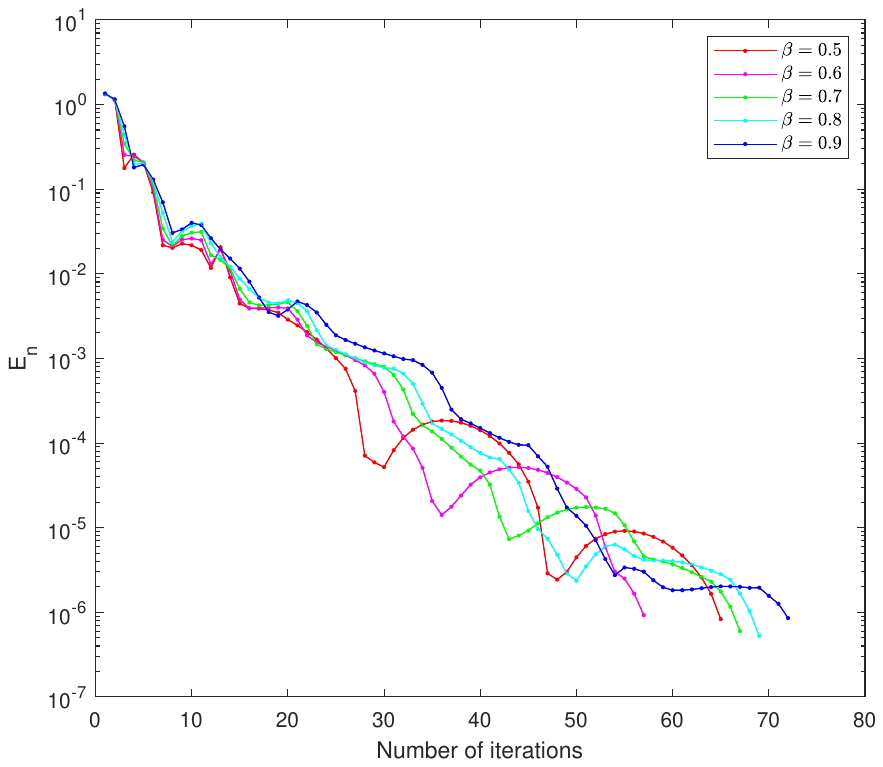}
  
		\caption{$\sigma=0.8$}
		\label{1}
	\end{subfigure}
	\begin{subfigure}{0.49\textwidth}
		\includegraphics[scale=0.5,trim=30mm 60mm 20mm 80mm]{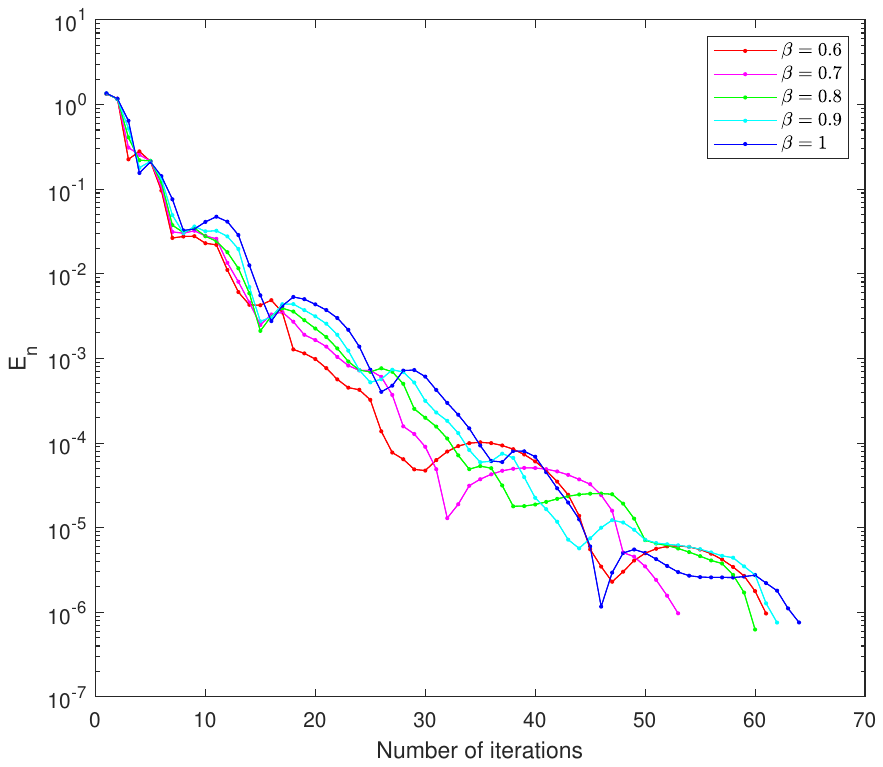}
		\caption{$\sigma=1$}
		\label{2}
	\end{subfigure}
	\begin{subfigure}{0.49\textwidth}
	\includegraphics[scale=0.5,trim=28mm 60mm 20mm 60mm]{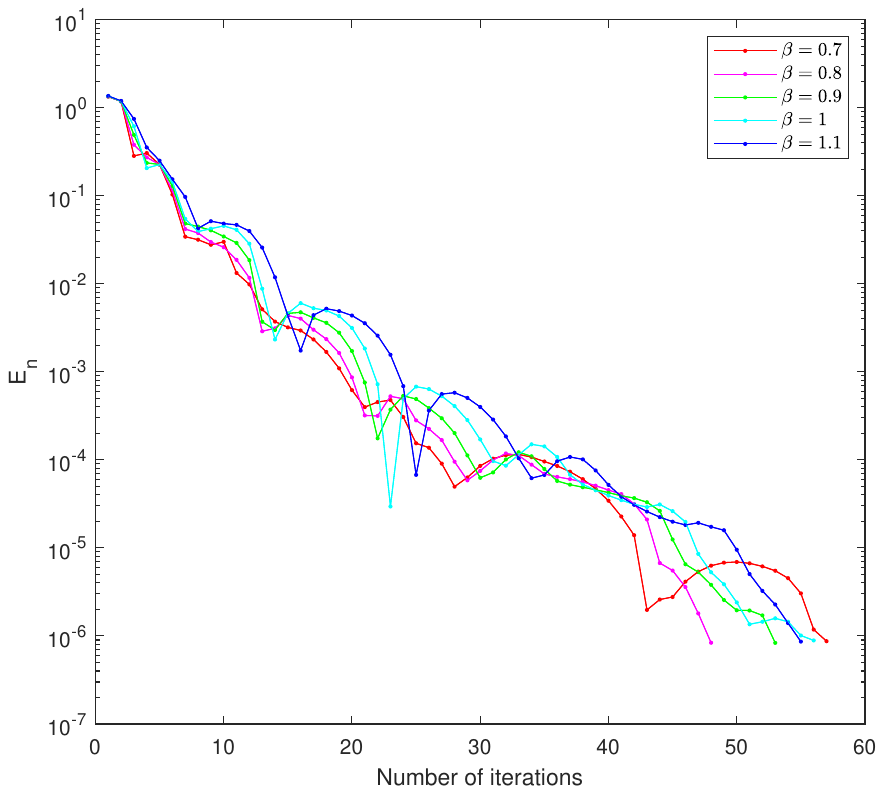}
	\caption{$\sigma=1.2$}
	\label{1}
\end{subfigure}
\begin{subfigure}{0.49\textwidth}
	\includegraphics[scale=0.5,trim=25mm 60mm 10mm 80mm]{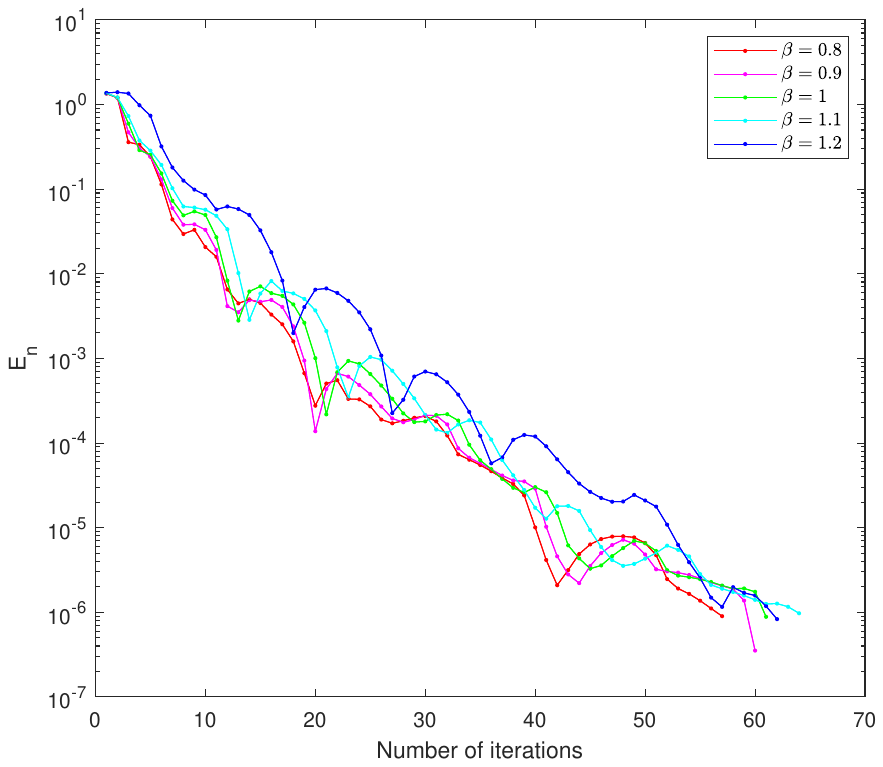}
	\caption{$\sigma=1.4$}
	\label{2}
\end{subfigure}
	\caption{Sensitivity analysis of Algorithm 4.1 (MDISEM) in network equilibrium flow with $\mu= 0.6$ and varying values of $\sigma$ and $\beta$.}
	\label{permemb}
\end{figure}

\subsection{Nash-Cournot oligopolistic market equilibrium model} \label{Nashhh1}
Here, we discuss the Nash-Cournot oligopolistic market equilibrium model, initially formulated as a convex optimization problem by Murphy et al. \cite{13}. A monotone variational inequality was developed from the model by Harker \cite{14}. For further details on this problem, we refer to \cite{15}. Consider $M$ firms supplying homogeneous products non-cooperatively. Let $g_{i}(x_{i})$ represent the cost of the $i$-th firm's supply, with $x_{i} \geq 0$. The total supply in the market is considered to be $R\geq0$, i.e. $R=\sum \limits_{i=1}^{M}x_{i}$ and $q(R)$ be the inverse demand curve. The Nash equilibrium solution for the market, denoted as $(p^{*}_{1},p^{*}_{2},...,p^{*}_{M})$, satisfies $p^{*}_{i} \geq 0$, for each $i=1,2,...,M$, and represents an optimal solution to the problem:
    \begin{align} \label{Nash1}
        \max_{x_{i}\geq 0}x_{i}q(x_{i}+R^{*}_{i})-g_{i}(x_{i})
    \end{align}
    where 
    \begin{align*}
        R^*_{i}=\sum \limits_{j=1,j\neq i}^{M}p^{*}_{j}.
    \end{align*}
    This problem \eqref{Nash1} is further equivalent to the VIP of finding $p^{*}=(p^{*}_{1},p^{*}_{2},...,p^{*}_{M}) \in \mathbb{R}^{M}_{+}$ such that
\begin{align}
    \langle \mathcal{F}p^{*},x-p^{*}\rangle \geq 0,
\end{align}
for each $x \in \mathbb{R}^{M}_{+}$. Here $\mathcal{F}p^{*}=(\mathcal{F}_{1}(p^{*}),\mathcal{F}_{2}(p^{*}),...,\mathcal{F}_{M}(p^{*}))$ and 
\begin{align*}
    \mathcal{F}_{i}(p^{*})=g'_{i}(p^{*}_{i})-q\bigg(\sum \limits_{j=1}^{M}p_{j}^{*}\bigg)-p^{*}_{j}q'\bigg(\sum \limits_{j=1}^{M}p^{*}_{j}\bigg).   
\end{align*}
 The cost function $g_{i}$ and the inverse demand curve structured as
$$g_{i}(x_{i})=e_{i}x_{i}+\frac{r_{i}}{r_{i}+1}O_{i}^{\frac{-1}{r_{i}}}x_{i}^{\frac{r_{i}+1}{r_{i}}}~\text{and}~q(R)=5000^{1/1.1}R^{-1/1.1},$$
where the parameters $e_{i},O_{i},r_{i}$ are shown in Table \ref{table:5a} below. Since $\mathcal{F}$ is monotone and Lipschitz continuous \cite{15}, it follows from Remark \ref{remarks}(iv) that the cost operator $\mathcal{F}$ satisfies Assumption $(A3)$.
Here, we compute for $M=5$ firms and the solution to the problem is $p^{*}=(36.912,41.842,43.705,42.665,39.182)$.  
We use the following parameters for the computation:
\begin{itemize}
    \item \textbf{MDISEM:} $\mu=0.6$, $\lambda_{1}=0.6$, $\beta=0.8$, $\sigma=1.5$, $\alpha_{n}=0.5$, $\delta_{n}=1+\frac{1}{n}$, $\chi_{n}=1+\frac{1}{(n+1)^{1.1}}$, $\zeta_{n}=\frac{1}{(n+1)^{1.1}}$, $\xi=0.4990$ and $\nu_{n}=1$.
    \item \textbf{L Algorithm:} $\mu=0.6$, $\lambda_{0}=0.6$ and $p_{n}=\frac{1}{(n+1)^{1.1}}$.
    \item \textbf{S Algorithm:} $\mu=0.6$, $\lambda_{1}=0.6$, $\alpha_{n}=0.2$, $\nu_{n}=1$ and $\gamma=1.9$.
    \item \textbf{T Algorithm:} $\mu=0.6$, $\nu_{1}=0.6$, $\theta=0.2$, and $\alpha_{n}=\frac{1}{n^{\frac{3}{2}}}$.
    
    \item \textbf{T1 Algorithm:} $\mu=0.6$, $\tau_{0}=0.6$, $\gamma=0.2$, $\theta=0.8$, $\beta=-0.6$, $\lambda=0.5$ and $\alpha_{n}=\frac{1}{(n+1)^{1.1}}$.
\end{itemize}
Based on Figure \ref{Nashgr} and Table \ref{Nashtab}, it is evident that the sequence produced by our algorithm converges to the solution much faster compared to previously established algorithms. Moreover, Figure \ref{permemb1} and Table \ref{sen2} provide a detailed sensitivity analysis for varying $\mu,\beta$ and $\sigma$. \\
From Table \ref{sen2}, the sensitivity analysis demonstrates that the iterative scheme reliably converges across varying values of 
$\mu$, $\beta$ and $\sigma$ ensuring stability in computing equilibrium quantities. Since the Nash-Cournot model involves firms adjusting their production strategies based on competitors’ actions, a fast and stable convergence of our algorithm implies that firms can reach an equilibrium state efficiently, minimizing computational overhead. 

\begin{table}[h!]
\centering
\begin{tabular}{ p{3cm} p{3cm} p{3cm} c }
\hline
Firm $i$ & $e_{i}$ & $O_{i}$ & $r_{i}$ \\
\hline
1 & 10 &5&1.2 \\
2 & 8   & 5&1.1 \\
3 &6 & 5&1.0 \\
4   &4 & 5&0.9 \\
5 & 2 & 5&0.8 \\
\hline
\end{tabular}
\caption{Parameters for the computation.}
\label{table:5a}
\end{table}
\vspace{2cm}

  \begin{table}[h!]
\centering
\begin{tabular}{ p{4cm} p{4cm} c }
\hline
Algorithm & CPU time (sec.)  & Iterations \\
\hline
MDISEM & 0.30 &80\\
L Algorithm & 0.49& 132\\
S Algorithm   &1.76 & 350 \\
T Algorithm &0.74 & 149\\
T1 Algorithm & 1.44 & 230\\
\hline
\end{tabular}
\caption{Numerical results for Nash-Cournot problem.}
\label{Nashtab}
\end{table}

\begin{figure} 
\centering
\includegraphics[scale=0.4,trim=36mm 35mm 30mm 25mm]{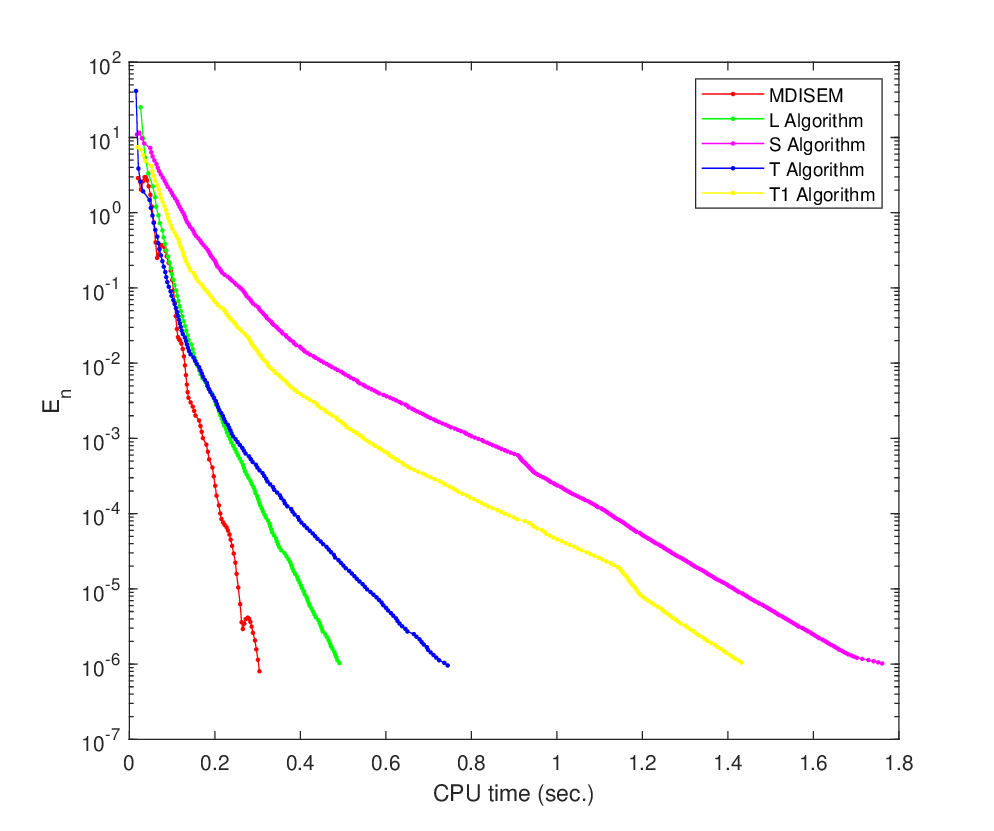}  
     \vspace{0.8cm}
    \caption{Comparison of algorithms for Nash-Cournot problem}
    \label{Nashgr}
\end{figure}

\begin{sidewaystable}[]
 	
 \footnotesize
 \begin{tabular*}{\textwidth}{ccccccccccccc}
 \toprule
  & \multicolumn{4}{c}{$\sigma=1.8$} & \multicolumn{4}{c}{$\sigma=4.9$} & \multicolumn{4}{c}{$\sigma=5.6$} \\
 \cmidrule(lr){2-5}\cmidrule(lr){6-9}\cmidrule(lr){10-13}
 $\mu=0.2323$ & $\beta=1.4$ & $\beta=2.6$ & $\beta=3.1$ & $\beta=4.2$ & $\beta=2.5$ & $\beta=3.1$ & $\beta=3.9$ & $\beta=4.1$ & $\beta=2.9$ & $\beta=3.3$ & $\beta=3.7$ & $\beta=4.01$ \\
 \midrule
 Iterations & 83 & 76 & 76 & 67 & 55 & 49 & 31 & 65 & 48 & 44 & 30 & 31 \\
 \bottomrule
 &&&&&&&&&&&&\\
 \toprule
  & \multicolumn{4}{c}{$\sigma=0.49$} & \multicolumn{4}{c}{$\sigma=1.21$} & \multicolumn{4}{c}{$\sigma=2.44$} \\
 \cmidrule(lr){2-5}\cmidrule(lr){6-9}\cmidrule(lr){10-13}
 $\mu=0.3332$ & $\beta=0.30$ & $\beta=1.1$ & $\beta=2.6$ & $\beta=2.8$ & $\beta=0.8$ & $\beta=1.2$ & $\beta=2.2$ & $\beta=2.7$ & $\beta=1.23$ & $\beta=1.4$ & $\beta=2.6$ & $\beta=3$ \\
 \midrule
 Iterations & 241& 182 & 96 & 97 & 92 & 82 & 78 & 77 & 81 & 83 & 78 & 74 \\
 \bottomrule
 &&&&&&&&&&&&\\
\toprule
  & \multicolumn{4}{c}{$\sigma=0.5$} & \multicolumn{4}{c}{$\sigma=1.8$} & \multicolumn{4}{c}{$\sigma=2.9$} \\
 \cmidrule(lr){2-5}\cmidrule(lr){6-9}\cmidrule(lr){10-13}
 $\mu=0.464$ & $\beta=0.3$ & $\beta=1.4$ & $\beta=1.9$ & $\beta=2.1$ & $\beta=1$ & $\beta=1.23$ & $\beta=1.96$ & $\beta=2.04$ & $\beta=1.56$ & $\beta=1.72$ & $\beta=1.89$ & $\beta=2.06$ \\
 \midrule
 Iterations & 203& 88 & 90 & 90 & 146 & 116 & 90 & 91 & 95 & 92 & 90 & 91 \\
 \bottomrule
 
 \end{tabular*}
 \label{table:sigma_beta13}

 \caption{Sensitivity analysis of Algorithm 4.1 (MDISEM) in Nash-Cournot problem for different values of $\sigma$ and $\beta$.}
 \label{sen2}
 \end{sidewaystable}

\begin{figure}[]
	\begin{subfigure}{0.49\textwidth}
		\includegraphics[scale=0.5,trim=26mm 60mm 20mm 80mm]{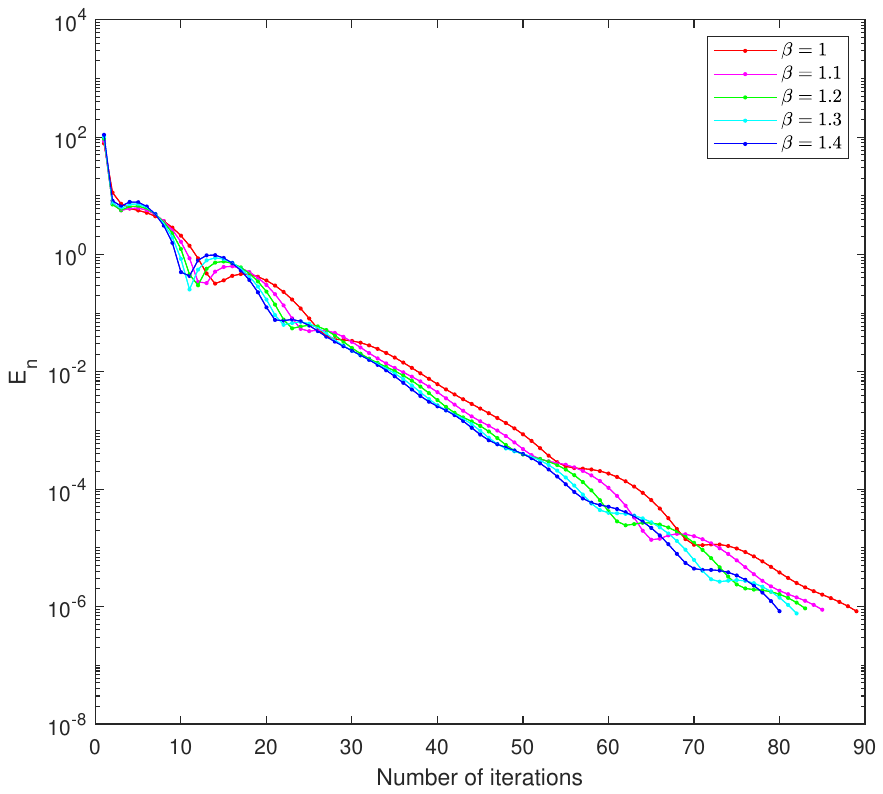}
		\caption{$\sigma=0.8$}
		\label{1}
	\end{subfigure}
	\begin{subfigure}{0.49\textwidth}
		\includegraphics[scale=0.5,trim=26mm 60mm 10mm 80mm]{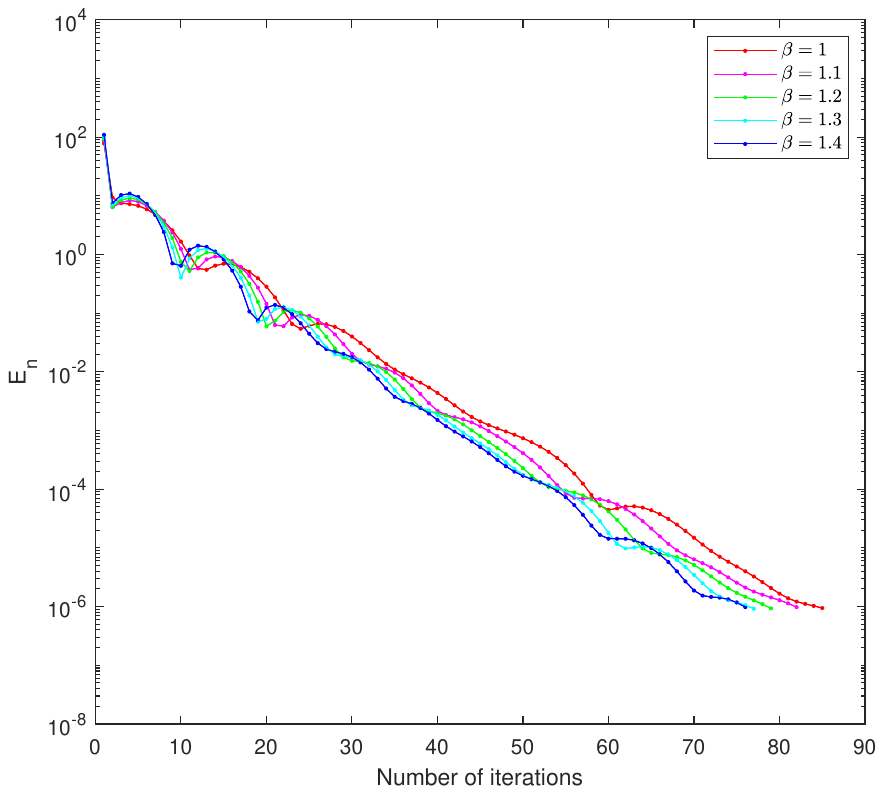}
		\caption{$\sigma=1$}
		\label{2}
	\end{subfigure}
	\begin{subfigure}{0.49\textwidth}
	\includegraphics[scale=0.5,trim=26mm 60mm 20mm 60mm]{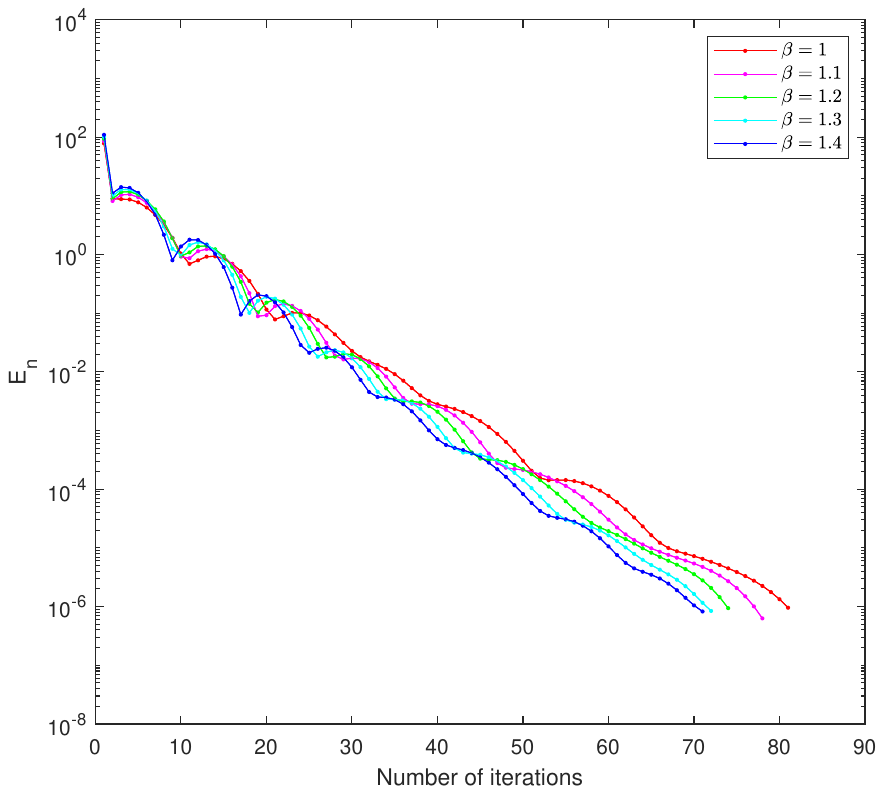}
	\caption{$\sigma=1.2$}
	\label{1}
\end{subfigure}
\begin{subfigure}{0.49\textwidth}
	\includegraphics[scale=0.5,trim=26mm 60mm 20mm 80mm]{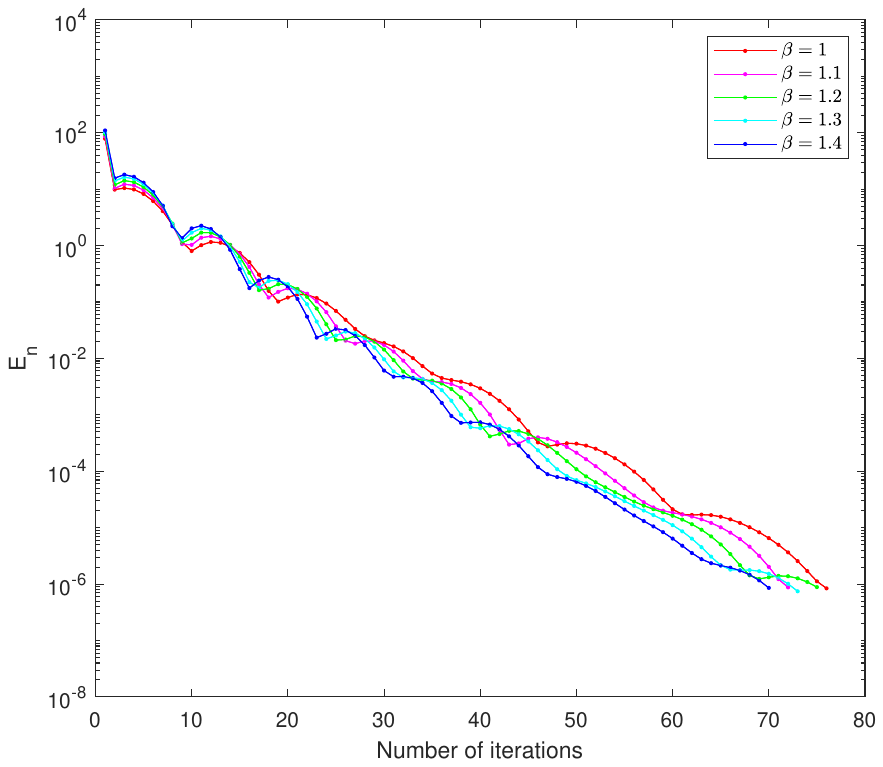}
	\caption{$\sigma=1.4$}
	\label{2}
\end{subfigure}
	\caption{Sensitivity analysis of Algorithm 4.1 (MDISEM) in Nash-Cournot problem with $\mu= 0.6$ and varying values of $\sigma$ and $\beta$.}
	\label{permemb1}
\end{figure}

\subsection{Image restoration problem}
In this section, we discuss the image restoration problems. 
The general model for image recovery can be formulated as:
\begin{align}
    b=Ax+v,
\end{align}
where $x \in \mathbb{R}^{n \times 1}$ is the original image, $A \in \mathbb{R}^{m \times n}$ is the blurring matrix, $v \in \mathbb{R}^{m \times 1}$ is the additive noise and $b \in \mathbb{R}^{m \times 1}$ is the observed image. In this case, we aim to approximate the original image by minimizing the additive noise known as a least square problem, which is as follows:
\begin{align} \label{imgrs1}
    \min_{x}\frac{1}{2}\|Ax-b\|^2.
\end{align}
This \eqref{imgrs1} is further equivalent to solving a variational inequality problem of finding $p^{*} \in \mathbb{R}^{n}$ such that
\begin{align}
    \nabla f(p^{*})^{T}(x-p^{*}) \geq 0, ~\text{for all}~x \in \mathbb{R}^{n},
\end{align}
where $\nabla f(x)=A^{T}(Ax-b)$. Now, $\mathcal{F}x-\mathcal{F}y=A^{T}A(x-y)$, thus $   \langle A^{T}(Ax-y),x-y \rangle=\langle A(x-y),A(x-y)\rangle \geq 0.$
It implies that $\nabla f$ is monotone. Also
\begin{align*}
    \|\mathcal{F}x-\mathcal{F}y\|&=\|A^{T}A(x-y)\|\\
    &\leq \|A^{T}A\|\|x-y\|.
\end{align*}
Hence $\nabla f$ is $\|A^{T}A\|$-Lipchitz continuous. Thus, $\nabla f$ satisfies the assumption $(A3)$. So, we apply our algorithm to restore the quality of images corrupted by different blurs. The test images are considered to be built-in Matlab pictures of the Cameraman (Figure \ref{figu7}) and Peppers (Figure \ref{figure1}). The following blur types are used to  corrupt the images:
\begin{enumerate}
    \item Gaussian blur of size 5 $\times$ 5 with standard deviation 1.5 on Cameraman (Figure \ref{figu8}).
    \item Motion blur with motion blur length 5 and motion blur angle 60 on peppers (Figure \ref{figure2}).
\end{enumerate}
We compute the relative error, $R_{n}=\frac{\|x_{n+1}-x_{n}\|}{x_{n}}$, and terminate the process once it falls below the predefined threshold $\epsilon$. We use the following parameters for the computation:
\begin{itemize}
    \item \textbf{MDISEM:} $\mu=0.6$, $\lambda_{1}=0.6$, $\beta=0.76$, $\sigma=1.5$, $\alpha_{n}=0.5$, $\delta_{n}=1+\frac{1}{n}$, $\chi_{n}=1+\frac{1}{(n+1)^{1.1}}$, $\zeta_{n}=\frac{1}{(n+1)^{1.1}}$, $\xi=0.4990$ and $\nu_{n}=0.4$.
    \item \textbf{L Algorithm:} $\mu=0.6$, $\lambda_{0}=0.6$ and $p_{n}=\frac{1}{(n+1)^{1.1}}$.
    \item \textbf{S Algorithm:} $\mu=0.6$, $\lambda_{1}=0.6$, $\alpha_{n}=0.2$, $\nu_{n}=1$ and $\gamma=1.9$.
    \item \textbf{T Algorithm:} $\mu=0.6$, $\nu_{1}=0.6$, $\theta=0.2$, and $\alpha_{n}=\frac{1}{(n+1)^{\frac{3}{2}}}$.
    \item \textbf{T1 Algorithm:} $\mu=0.6$, $\tau_{0}=0.2$, $\gamma=1.1$, $\theta=0.3$, $\beta=-0.2$, $\lambda=0.5$ and $\alpha_{n}=\frac{1}{(n+1)^{1.1}}$.
\end{itemize}
From Figure \ref{Imgfig} and Table \ref{table:img}, it is evident that our algorithm outperforms previously established results. The recovered images are shown in Figures \ref{imageres1} and \ref{imageres2}.
\begin{table}[h!] 
\centering
\begin{tabular}{p{4cm} c c c c}
\hline
\multirow{2}{*}{Algorithm} & \multicolumn{2}{c}{Gaussian blur} & \multicolumn{2}{c}{Motion blur} \\
\cline{2-5}
& CPU time (sec.) & Iterations & CPU time (sec.) & Iterations\\
\hline
MDISEM & 2.12 & 16 & 10.44 & 9 \\
L Algorithm & 2.82 & 23 &12.17& 13 \\
S Algorithm & 2.75 & 24& 10.71 & 11  \\
T Algorithm & 2.82 & 26 & 10.78 & 11  \\
T1 Algorithm &2.83&20& 15.60&13\\
\hline
\end{tabular}
\caption{Numerical results for different blur types.}
\label{table:img}
\end{table}

\begin{figure}[]
	\begin{subfigure}{0.49\textwidth}
		\includegraphics[scale=0.5,trim=10mm 60mm 20mm 20mm]{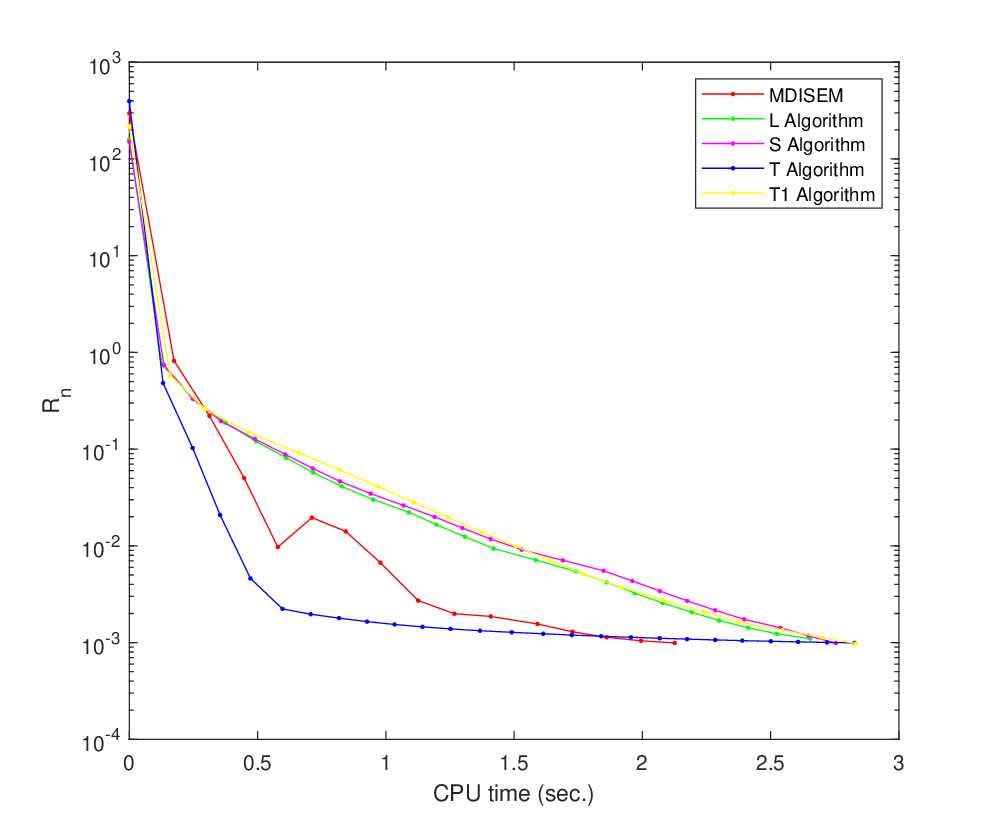}
        \vspace{1.2in}
		\caption{Gaussian blur with $\epsilon=10^{-3}$.}
		\label{1gaussian}
	\end{subfigure}
	\begin{subfigure}{0.49\textwidth}
		\includegraphics[scale=0.5,trim=10mm 60mm 10mm 20mm]{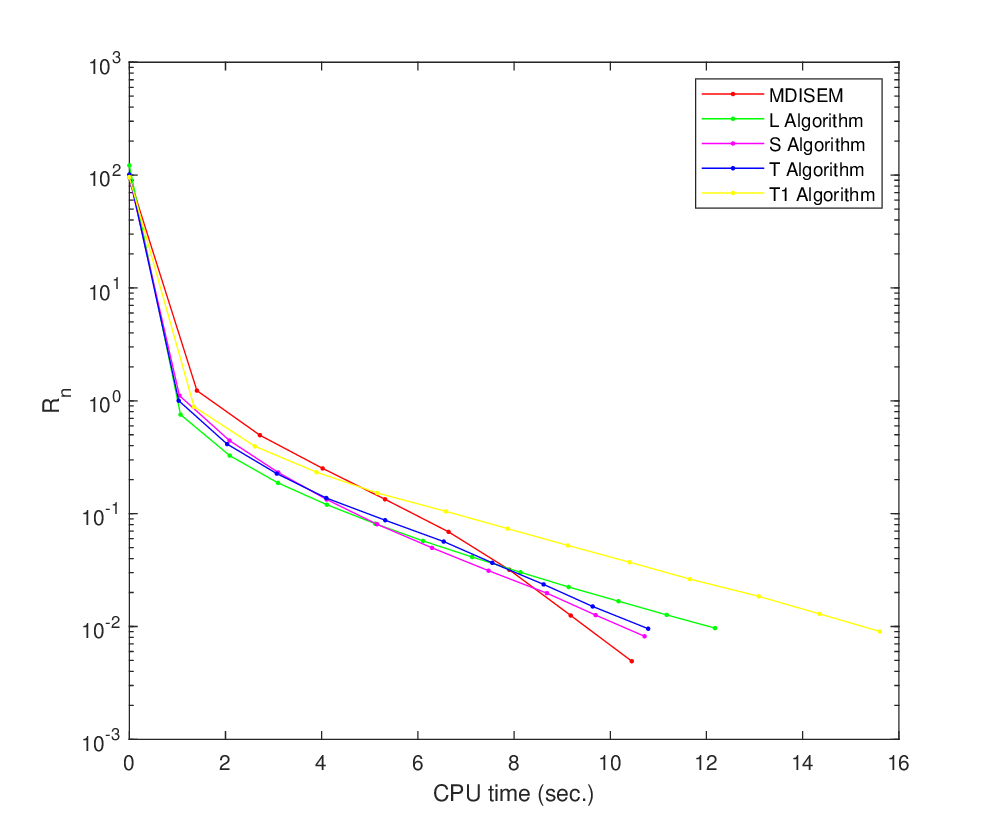}
        \vspace{1.2in}
		\caption{Motion blur with $\epsilon=10^{-2}$. }
		\label{2motion}
	\end{subfigure}
    \caption{Comparison of algorithms in image restoration problem}
    \label{Imgfig}
    \end{figure}
\begin{figure} \label{im11}
    \centering
    \begin{subfigure}{0.32\textwidth}
        \includegraphics[width=\linewidth, trim=2mm 1mm 1mm 1mm]{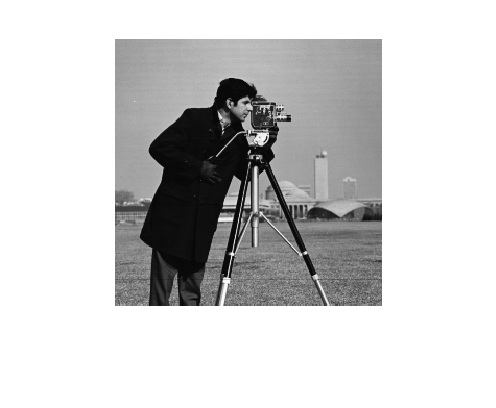}
        \vspace{-0.6in}
        \caption{Original}
        \label{figu7}
    \end{subfigure}
    \begin{subfigure}{0.32\textwidth}
        \includegraphics[width=\linewidth, trim=2mm 1mm 1mm 1mm]{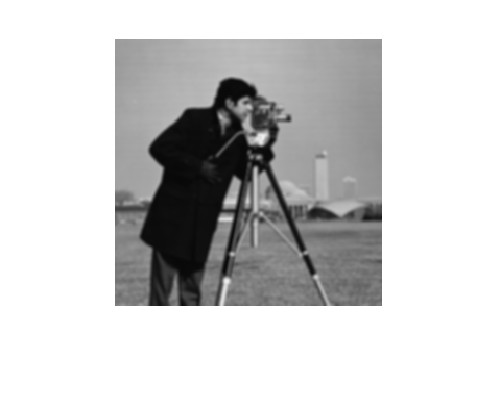}
        \vspace{-0.6in}
        \caption{Gaussian Blur}
        \label{figu8}
    \end{subfigure}
    \begin{subfigure}{0.32\textwidth}
        \includegraphics[width=\linewidth, trim=2mm 1mm 1mm 1mm]{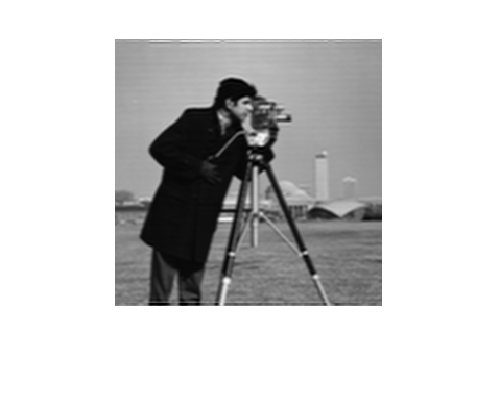}
        \vspace{-0.6in}
        \caption{MDISEM}
        \label{figu8main}
    \end{subfigure}
    
    \begin{subfigure}{0.32\textwidth}
        \includegraphics[width=\linewidth, trim=2mm 1mm 1mm 1mm]{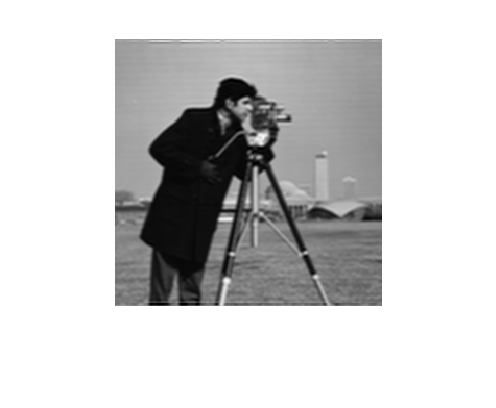}
        \vspace{-0.6in}
        \caption{L Algorithm}
        \label{figure4y}
    \end{subfigure}
    \begin{subfigure}{0.32\textwidth}
        \includegraphics[width=\linewidth, trim=2mm 1mm 1mm 1mm]{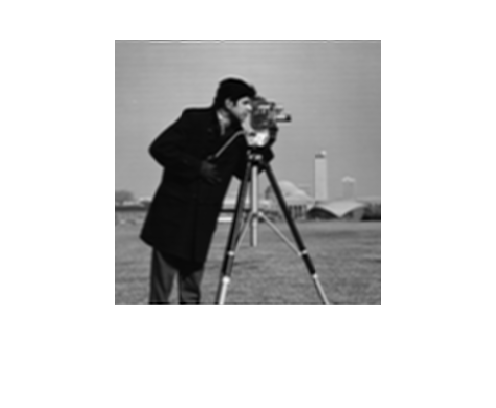}
        \vspace{-0.6in}
        \caption{S Algorithm}
        \label{figu10}
    \end{subfigure}
    \begin{subfigure}{0.32\textwidth}
        \includegraphics[width=\linewidth, trim=2mm 1mm 1mm 1mm]{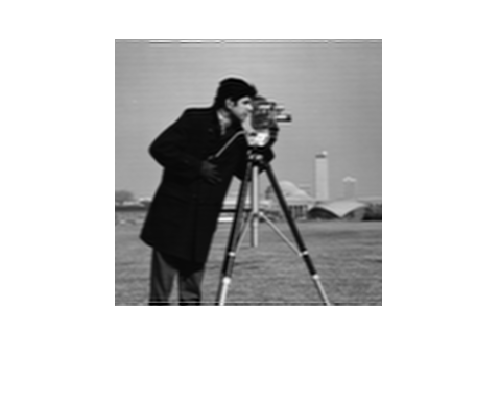}
        \vspace{-0.6in}
        \caption{T Algorithm}
        \label{figu9}
    \end{subfigure}
    
    \begin{subfigure}{0.32\textwidth}
        \centering
        \includegraphics[width=\linewidth, trim=2mm 1mm 1mm 1mm]{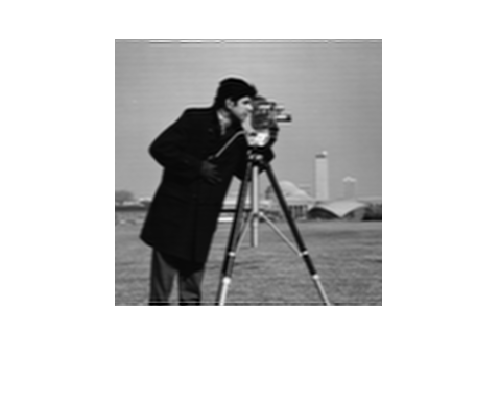}
        \vspace{-0.6in}
        \caption{T1 Algorithm}
        \label{figu10a}
    \end{subfigure}

    \caption{Image restoration from Gaussian blur.}
    \label{imageres1}
\end{figure}

\begin{figure}
    \centering
    \begin{subfigure}{0.32\textwidth}
        \includegraphics[width=\linewidth, trim=2mm 1mm 1mm 1mm]{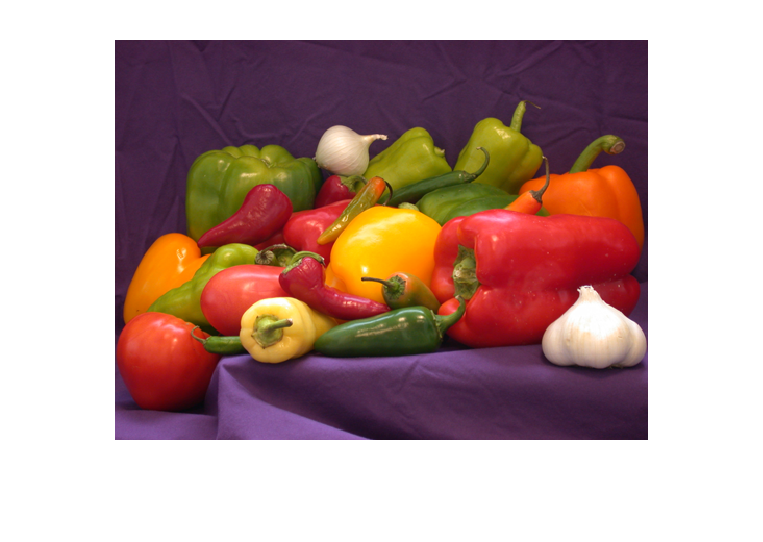}
        \vspace{-0.5in}
        \caption{Original}
        \label{figure1}
    \end{subfigure}%
    \begin{subfigure}{0.32\textwidth}
        \includegraphics[width=\linewidth, trim=2mm 1mm 1mm 1mm]{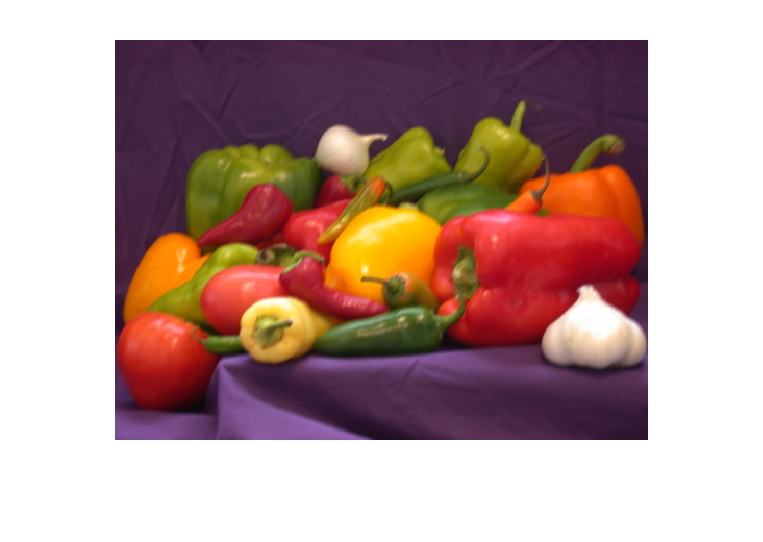}
        \vspace{-0.5in}
        \caption{Motion blur}
        \label{figure2}
    \end{subfigure}%
    \begin{subfigure}{0.32\textwidth}
        \includegraphics[width=\linewidth, trim=2mm 1mm 1mm 1mm]{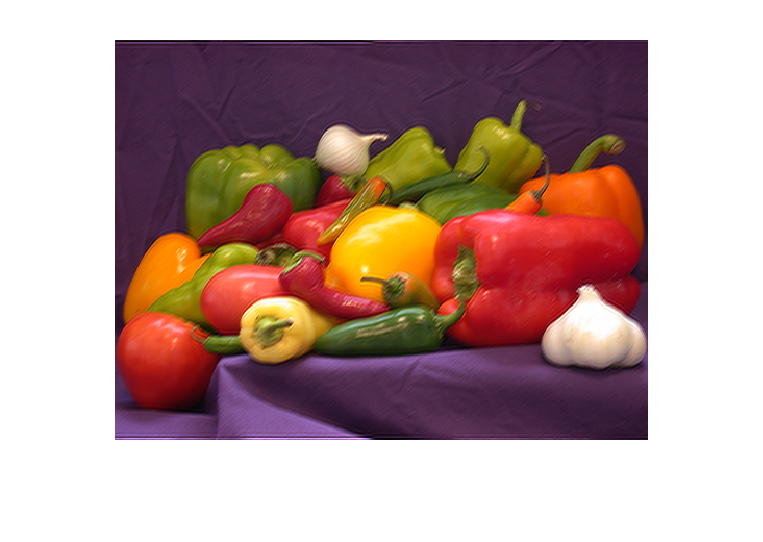}
        \vspace{-0.5in}
        \caption{MDISEM}
        \label{figure3}
    \end{subfigure}

    \begin{subfigure}{0.32\textwidth}
        \includegraphics[width=\linewidth, trim=2mm 1mm 1mm 1mm]{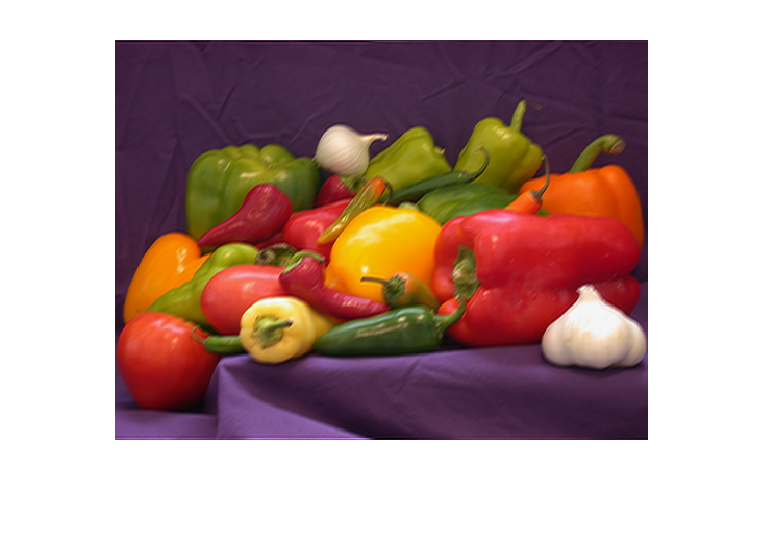}
        \vspace{-0.5in}
        \caption{L Algorithm}
        \label{figu10}
    \end{subfigure}%
    \begin{subfigure}{0.32\textwidth}
        \includegraphics[width=\linewidth, trim=2mm 1mm 1mm 1mm]{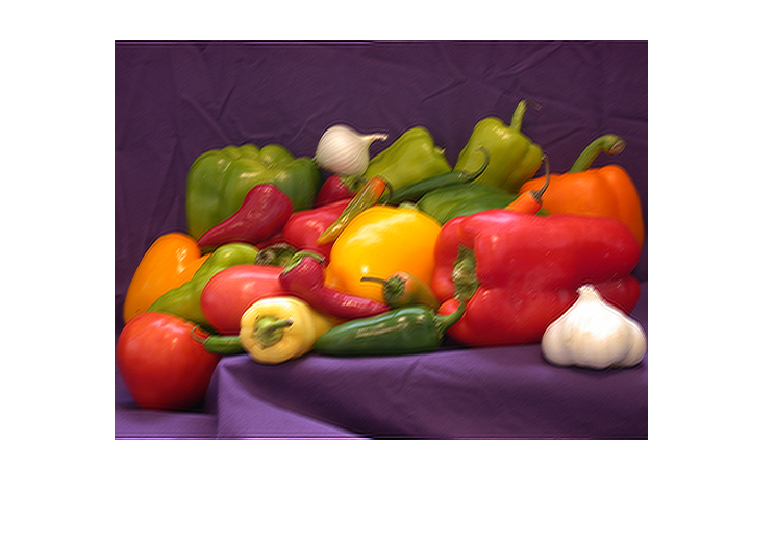}
        \vspace{-0.5in}
        \caption{S Algorithm}
        \label{figure5}
    \end{subfigure}%
    \begin{subfigure}{0.32\textwidth}
        \includegraphics[width=\linewidth, trim=2mm 1mm 1mm 1mm]{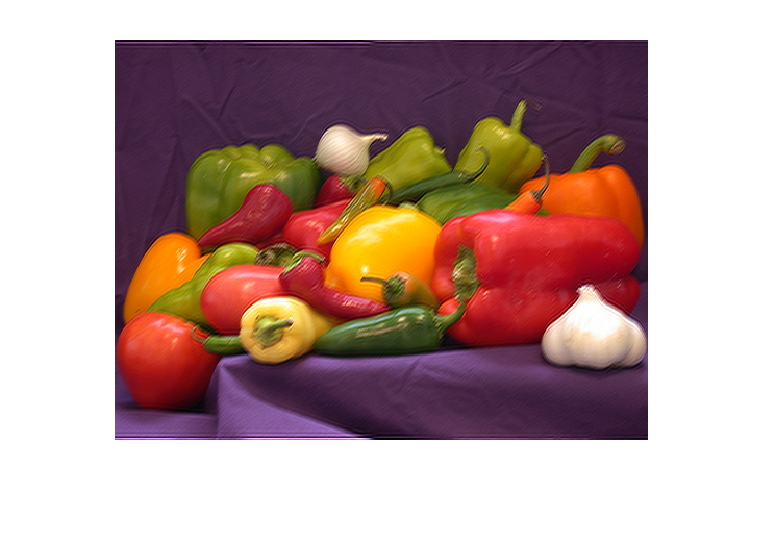}
        \vspace{-0.5in}
        \caption{T Algorithm}
        \label{figure4}
    \end{subfigure}

    \begin{subfigure}{0.32\textwidth}
        \centering
        \includegraphics[width=\linewidth, trim=2mm 1mm 1mm 1mm]{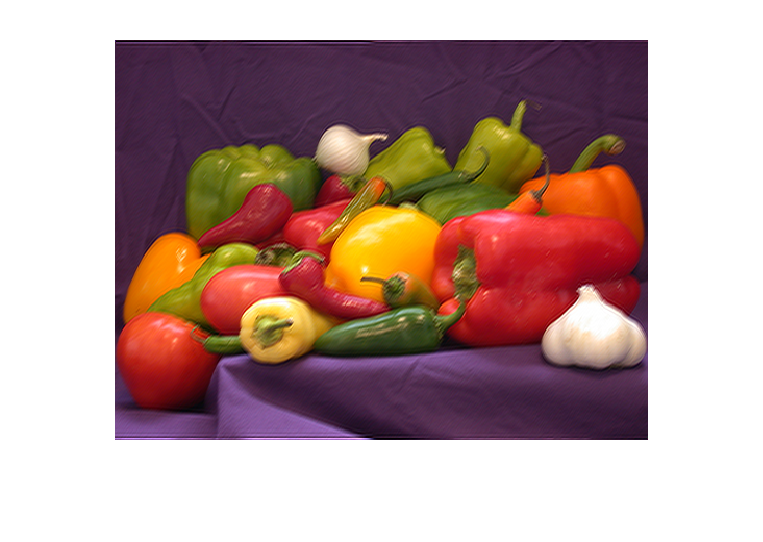}
        \vspace{-0.5in}
        \caption{T1 Algorithm}
        \label{figure5a}
    \end{subfigure}

    \caption{Image restoration from motion blur.}
    \label{imageres2}
\end{figure}
 \section{Conclusion}
        In this paper, we introduced a new efficient iterative algorithm to solve the variational inequalities in the setting of real Hilbert space. The proposed algorithm is motivated by the use of the double inertial method, in which one of the inertial is allowed to be 1. Moreover, our iterative algorithm chooses a generalized step-size that is non-monotonic. Finally, we give some real-life applications for network equilibrium flow, oligopolistic market equilibrium problems and image restoration problems. We also concluded that our iterative scheme works much better in terms of computational time and converges to a solution in less number of iterations. For future work, we aim to extend this work in the setting of reflexive Banach spaces.

\textbf{Acknowledgements.} The authors are thankful to the learned referees for the valuable suggestions and appreciation of the work.
The first author is also grateful to CSIR, New Delhi, India, for providing a senior research fellowship (File 09/0677(13166)/2022-EMR-I).

\section*{Declarations}
\textbf{Ethical Approval.} Not Applicable as no both human and/ or animal studies. 

\textbf{Data Availability.} No underlying data was collected or produced in this study.

\textbf{Competing interests.}
The authors declare that there is no competing interest in the publication of this paper.

\textbf{Funding.} There is no funding.

\textbf{Authors' contributions.} All authors contribute equally.


\end{document}